\documentclass[a4paper, 11pt, twoside]{amsart}
\usepackage[a4paper, top=3cm, bottom=3cm, right=2.5cm, left=2.5cm]{geometry}

\usepackage[english]{babel}
\usepackage[utf8]{inputenc}
\usepackage[T1]{fontenc}
\usepackage{microtype}
\usepackage[style=british]{csquotes}

\usepackage{amsmath, amsthm, amssymb, bm, mathtools}
\usepackage{enumerate}

\usepackage{calc}
\usepackage{graphicx}
\graphicspath{{img/}}

\usepackage[%
	unicode,
	colorlinks=true,
	linkcolor=blue,
	urlcolor=blue,
	linktocpage=true,
	citecolor=blue]{hyperref}
\hypersetup{%
	pdfpagemode = UseNone,
	pdfstartview = FitH,
	pdfdisplaydoctitle = true,
	pdfborder = {0 0 0}
}

\emergencystretch=\maxdimen
\makeatletter
\renewcommand\subsection{\@startsection{subsection}{2}%
  \z@{-.5\linespacing\@plus-.7\linespacing}{.5\linespacing}%
  {\normalfont\scshape}}
\makeatother

\newtheorem{thm}{Theorem}[section]
\newtheorem{lemma}[thm]{Lemma}
\newtheorem{prop}[thm]{Proposition}
\newtheorem{cor}[thm]{Corollary}

\theoremstyle{definition}

\theoremstyle{remark}
\newtheorem{remark}[thm]{Remark}

\numberwithin{equation}{section}

\def\P{\mathbb{P}}
\def\H{\mathbb{H}}
\def\E{\mathbb{E}}
\def\R{\mathbb{R}}

\def\D{\mathbb{D}}

\DeclareMathOperator{\SLE}{SLE}
\DeclareMathOperator{\CLE}{CLE}
\DeclareMathOperator{\BCLE}{BCLE}
\DeclareMathOperator{\capacity}{cap\,}
\DeclareMathOperator{\hcap}{hcap\,}

\begin{document}

\title{The trunks of CLE(4) explorations}
\author{Matthis Lehmkuehler}
\date{July 12, 2021}
\begin{abstract}
	A natural class of conformally invariant ways for discovering the loops of a conformal loop ensemble $\CLE_4$ is given by a certain family of $\SLE_4^{\langle\mu\rangle} (-2)$ exploration processes for $\mu \in \R$. Such an exploration consists of one simple continuous path called the trunk of the exploration that discovers $\CLE_4$ loops along the way. The parameter $\mu$ appears in the Loewner chain description of the path that traces the trunk and all $\CLE_4$ loops encountered by the trunk in chronological order.  These explorations can also be interpreted in terms of level lines of a Gaussian free field. 

	It has been shown by Miller, Sheffield and Werner that the trunk of such an exploration is an $\SLE_4(\rho,-2-\rho)$ process for some (unknown) value of $\rho \in (-2, 0)$. The main result of the present paper is to establish the relation between $\mu$ and $\rho$, more specifically to show that $\mu = -\pi\cot(\pi\rho/2)$. 
	
	The crux of the paper is to show how explorations of $\CLE_4$ can be approximated by explorations of $\CLE_\kappa$ for $\kappa \uparrow 4$, which then makes it possible to use recent results by Miller, Sheffield and Werner about the trunks of $\CLE_\kappa$ explorations for $\kappa < 4$.
\end{abstract}
\let\thefootnote\relax\footnote{M.\,L. -- ETH Zürich, Rämistrasse 101, 8092 Zürich, Switzerland}

\maketitle
\tableofcontents
\thispagestyle{empty}
\newpage

\section{Introduction}
\label{sec:introduction}

\subsection{Description of our main result} 
The conformal loop ensemble $\CLE_4$ is a random collection of disjoint simple loops in a simply connected domain in the plane which is of particular interest in random geometry. It is conjectured to be related to the scaling limit of critical Potts models with $q=4$ colors and is known to be closely related to the Gaussian free field (GFF). More specifically, for a well-chosen particular value of $\lambda$, when one considers a GFF with Dirichlet boundary conditions in a simply connected subset of the plane, then the outermost level lines of level $\pm \lambda$ (i.e. interfaces between domains with $0$ and $\pm 2\lambda$ boundary conditions) in the sense of Schramm-Sheffield (see \cite{ss-contour, dubedat-coupling}) form a $\CLE_4$ (this is a result by Miller and Sheffield, see \cite{asw-btls} for a self-contained treatment). It has also been shown that $\CLE_4$ can be constructed via critical Brownian loop soup clusters (see \cite{shef-werner-cle, werner-qian-bls}, see also \cite{qian-bls-conditioned} for further loop soup results related to the CLE explorations that we will be discussing). CLEs can be constructed via variants of SLE as proposed in \cite{shef-cle}. 

If one removes the interiors of all the loops of a simple CLE, one gets a random fractal set, which can be viewed as some random conformally invariant analog of the Sierpinski carpet, and is often referred to as the CLE carpet. One of the important general features of the Conformal Loop Ensembles $\CLE_\kappa$ for $\kappa \in (8/3, 4]$  are their conformal restriction properties as introduced and studied in \cite{shef-werner-cle}. This leads very naturally to look for conformally invariant ways to explore and discover the CLE loops one after the other when starting from the boundary. The aforementioned SLE variants from \cite{shef-cle} form examples of such exploration mechanisms. The main topic of the present paper is to discuss an exact identity between two ways of constructing such exploration processes in the case of $\CLE_4$.

Before stating our main result, let us provide some background about $\CLE_4$ explorations. We will start with some rather informal heuristic considerations: 
In the case of $\CLE_4$, it turns out to be very natural (and actually in some sense necessary in order to discover the $\CLE_4$ loops $\{\gamma_i\colon i\ge 1\}$ along some continuous and conformally invariant curve in the CLE carpet) to assign to each $\CLE_4$ loop $\gamma_i$ the outcome $\sigma_i \in \{ \pm 1 \}$ of an independent fair coin toss. The collection $\{(\gamma_i, \sigma_i)\colon i \ge 1\}$ is then called a labeled $\CLE_4$ and denoted by $\CLE_4^0$ (the role of this additional randomness will be also clear in the GFF level line perspective that we will recall in a moment). Then, it follows from \cite{cle-percolations} that there exists a natural one-parameter family of Markovian and conformally invariant exploration curves that discover a labeled $\CLE_4$ with the following features: The exploration consists of a continuous simple curve $\eta$ from one given boundary point to another (for instance, from $-i$ to $i$ in the unit disk). This continuous curve is called the trunk of the exploration; it stays in the $\CLE_4$ carpet, but it hits many $\CLE_4$ loops along the way. One can then consider the continuous curve $\gamma$ that is obtained by tracing all loops encountered by this trunk in chronological order (one traces the loop $\gamma_i$ clockwise or anticlockwise when $\sigma_i=-1$ or $\sigma_i=1$, respectively). The curve $\gamma$ is then a continuous curve from the starting point to the endpoint of $\eta$ with many double points (because the trunk can hit a CLE loop many times).

One way to interpret the trunk $\eta$ goes as follows: We can consider a $\CLE_4$ loop $\gamma_i$ (and its inside) to be open or closed depending on whether $\sigma_i=1$ or $\sigma_i=-1$. Then, even though the $\CLE_4$ loops are disjoint, there exists a one-parameter family of procedures for deterministically and in a conformally invariant way agglomerating the loops into open (and closed) clusters. The trunk then traces the interface between the closed clusters touching the clockwise boundary segment from $-i$ to $i$, and the open clusters touching the counterclockwise boundary segment from $-i$ to $i$. The parameter parameterizing this one-parameter family in some sense encodes how strong the glue is that sticks the open $\CLE_4$ loops together into clusters.

Even if $\gamma$ is not a simple curve, it can nevertheless be viewed as the continuous curve that generates a Loewner chain, and it is quite easy to guess the form of the law of its driving function $\xi$ because of the Markovian nature of the exploration.

As proposed already in \cite{shef-cle}, the candidate for the law of $\gamma$ should be a $\SLE_4^{\langle \mu \rangle}(-2)$ process for $\mu \in \R$. The constant $\mu$ that enters in the definition of the driving function of $\gamma$ intuitively corresponds to a `drift choice' within the carpet. The larger $\mu$ is, the more the discovery process is `pushed to the right' when it is in the carpet.

We now move to the more rigorous construction of these exploration processes. 
The precise definition of the driving function $\xi$ goes as follows: Let $B$ be a standard Brownian motion, then (see for instance \cite {donati-yor-bessel,yor-bm-eth}) it is possible to define its principal value process by
$ P := \lim_{\epsilon\to 0}\int_0^{\,\cdot} 1(|B_t|\ge \epsilon)\, dt / {B_t}$. 
Then, if $\ell$ denotes the local time at $0$ of $B$, we can define for each given $\mu\in \R$, the processes 
$O = \mu \ell - P$ and $\xi = 2B + O$. This continuous process $\xi$ is then the candidate for the driving function of $\gamma$. 

The steps in the construction of $\gamma$ and $\eta$ then go like this: Given $\mu$, one can \emph{define} $\SLE_4^{\langle \mu \rangle}(-2)$ as a Loewner chain driven by the continuous function $\xi$,  and one then shows that this Loewner chain is almost surely generated by a random curve $\gamma$ that can indeed be (deterministically) decomposed into a trunk $\eta$ and the loops of a $\CLE_4$ that this trunk hits (see \cite{cle-percolations} and the references therein). 

A further observation is that the maximal (open) intervals of the parametrization of $\gamma$ on which $\gamma$ does not have a double point yield a family $\Gamma_-$ of clockwise oriented loops and a family $\Gamma_+$ of counterclockwise oriented loops; the trunk $\eta$ of $\gamma$ is then the unique curve which lies right of all loops in $\Gamma_-$ and left of all loops in $\Gamma_+$ (so that $\eta$ is a deterministic function of $\gamma$). In fact, $\Gamma_-\cup \Gamma_+$ are (by definition) precisely the loops discovered by the exploration path $\gamma$ in a $\CLE_4$. One can take this as the definition of the trunk $\eta$ of $\gamma$ (and this definition extends more generally to explorations of simple CLEs).

The positive (resp. negative) excursions of $B$ correspond to CLE loops in $\Gamma_+$ (resp. CLE loops in $\Gamma_-$) being traced by $\gamma$. As is apparent from the definition, the value of $\mu$ only effects the evolution of $(\xi,O)$ at times when $\xi-O$ is zero. Heuristically, the local time push in the definition of $O$ corresponds to the point from which we continue the exploration path being shifted by an infinitesimal amount along the explored hull whenever we have just completed tracing a CLE loop. This asymmetry in the law of $\gamma$ results in an asymmetry in the law of its trunk $\eta$.

One further result of \cite{cle-percolations} is that the law of the trunk (alone) is then necessarily of the following type -- which is a posteriori not so surprising in view of its properties: 
 
\begin{thm}[\cite{cle-percolations}]
	\label{thm:main-result-msw}
	For each $\mu \in \R$,  if $\gamma$ is an $\SLE_4^{\langle \mu \rangle}(-2)$, then the law of its trunk $\eta$ is that of an $\SLE_4 ( \rho, -2 -\rho)$ for some $\rho = R( \mu) \in (-2, 0)$.  Furthermore, $R$ is an increasing bijection from 
	$\R$ onto $(-2, 0)$ such that  $R ( - \mu) = -2 - R (\mu)$.
\end{thm}
We write $M=R^{-1}$. This raises naturally the question of what this function $R$ (or its inverse function $M$) is. The main purpose of this paper is precisely to answer that question: 
\begin{thm} 
	\label{thm:main-result}
	The relation between $\mu \in \R$ and $\rho \in (-2, 0)$ in Theorem \ref {thm:main-result-msw} is $\mu = -\pi\cot(\pi\rho/2)$.
\end{thm}

This formula relating $\mu$ to $\rho$ looks fairly simple, but our proof will rely on a number of rather intricate ideas. Indeed, we will obtain this relation by considering the $\kappa \uparrow 4$ limit of recent results by Miller, Sheffield and Werner \cite{cle-percolations, msw-simple} that deal with conformally invariant explorations of $\CLE_\kappa$ for $\kappa \in (8/3, 4)$. Those result in turn build on the features of such conformally invariant explorations, when one equips the $\CLE_\kappa$ with an independent Liouville quantum gravity (LQG) structure on it.

In the present paper, we will however not discuss any LQG aspects. Indeed, almost all the work will consist in carefully studying what happens to these explorations and their trunks in appropriately chosen $\kappa \uparrow 4$ limits. LQG features will therefore only be used indirectly (since they are a fundamental tool in deriving the results from \cite{msw-simple} that we will use at the very end of the paper). We can note that this is one occurrence of the fact that while $\CLE_4$ is in a number of ways simpler than $\CLE_\kappa$ for $\kappa < 4$ (for instance, because of its direct interpretation in terms of GFF level lines, or because the $\CLE_4$ explorations are deterministic functions of the labeled $\CLE_4$ -- a result which fails to be true for $\CLE_\kappa$ explorations with $\kappa < 4$, see \cite{msw-sle-range}), the LQG technology is trickier to handle for $\kappa =4$ than for $\kappa < 4$, which explains this rather convoluted-looking derivation of our formulas.

\subsection{Reformulation in terms of GFF level lines}

\begin{figure}
	\centering
	\def\svgwidth{0.8\columnwidth}
	%% Creator: Inkscape 1.1 (c4e8f9ed74, 2021-05-24), www.inkscape.org
%% Accompanies image file 'gff-coupling.pdf' (pdf, eps, ps)
\begingroup%
  \makeatletter%
  \providecommand\color[2][]{%
    \errmessage{(Inkscape) Color is used for the text in Inkscape, but the package 'color.sty' is not loaded}%
    \renewcommand\color[2][]{}%
  }%
  \providecommand\transparent[1]{%
    \errmessage{(Inkscape) Transparency is used (non-zero) for the text in Inkscape, but the package 'transparent.sty' is not loaded}%
    \renewcommand\transparent[1]{}%
  }%
  \providecommand\rotatebox[2]{#2}%
  \newcommand*\fsize{\dimexpr\f@size pt\relax}%
  \newcommand*\lineheight[1]{\fontsize{\fsize}{#1\fsize}\selectfont}%
  \ifx\svgwidth\undefined%
    \setlength{\unitlength}{307.33231661bp}%
    \ifx\svgscale\undefined%
      \relax%
    \else%
      \setlength{\unitlength}{\unitlength * \real{\svgscale}}%
    \fi%
  \else%
    \setlength{\unitlength}{\svgwidth}%
  \fi%
  \global\let\svgwidth\undefined%
  \global\let\svgscale\undefined%
  \makeatother%
  \begin{picture}(1,0.99294795)%
    \lineheight{1}%
    \setlength\tabcolsep{0pt}%
    \put(0,0){\includegraphics[width=\unitlength,page=1]{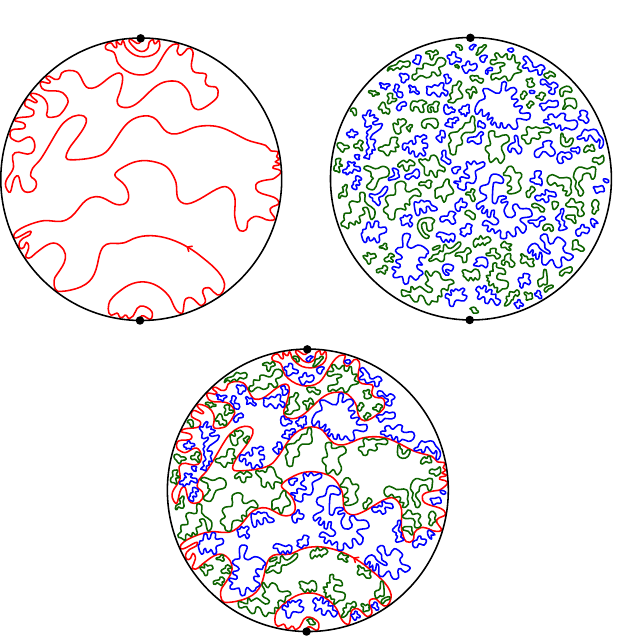}}%
    \put(0.39628908,0.53465272){\color[rgb]{0,0,0}\makebox(0,0)[lt]{\lineheight{1.25}\smash{\begin{tabular}[t]{l}$\eta$\end{tabular}}}}%
    \put(0.91092539,0.52303609){\color[rgb]{0,0,0}\makebox(0,0)[lt]{\lineheight{1.25}\smash{\begin{tabular}[t]{l}$\Gamma$\end{tabular}}}}%
    \put(0.65817621,0.03691979){\color[rgb]{0,0,0}\makebox(0,0)[lt]{\lineheight{1.25}\smash{\begin{tabular}[t]{l}$\gamma$\end{tabular}}}}%
    \put(0,0){\includegraphics[width=\unitlength,page=2]{gff-coupling.pdf}}%
  \end{picture}%
\endgroup%

	\caption{\emph{Top left.} The red curve $\eta$ is the level line of $h$ from $-i$ to $i$ at level $c\in (-\lambda,\lambda)$ (i.e. there are $c-\lambda$ boundary conditions left of $\eta$, there are $c+\lambda$ boundary conditions right of $\eta$ and there are $0$ boundary conditions on the boundary of the domain $\partial \D$). \emph{Top right.} The $\pm\lambda$ level lines $\Gamma$ of the Dirichlet GFF $h$ are shown. Blue (resp. green) loops correspond to $2\lambda$ (resp. $-2\lambda$) GFF boundary conditions on the inside of the loops. \emph{Bottom.} The curve $\eta$ does not intersect the interior of any of the loops in $\Gamma$; by removing all loops not intersecting $\eta$ we obtain the loops in the figure and chronologically attaching them to $\eta$ yields a curve $\gamma$. Note that the loops touching $\eta$ on the right (resp. left) have positive (resp. negative) boundary conditions on the inside.}
	\label{fig:gff-coupling}
\end{figure}

There is a simple way to describe both $\gamma$ and $\eta$ using the level lines of a Dirichlet GFF that we now briefly review.
This also sheds light on how to describe the conditional law of $\gamma$ given its trunk (see \cite{cle-percolations}).

Recall first that the GFF in the unit disk is a random Gaussian distribution $h_0$ on the unit disk $\D$ with covariance kernel $G(z,w) = -\log (|z-w| / |1-\bar{z}w|)$. Note that $G$ is the Dirichlet Green's function on $\D$. When $\phi\colon \partial\D\to \R$ is measurable and bounded, we call the law of $h=h_0+H_\phi$ a Gaussian free field with boundary condition $\phi$ where $H_\phi$ is the unique harmonic extension of $\phi$ to $\D$. Importantly, the field $h_0$ satisfies conformal invariance which allows us to extend these definitions to arbitrary proper simply connected domains and disjoint unions of such domains.

While $h$ is a distribution and not defined pointwise, there is a way to deterministically construct level lines from a Gaussian free field \cite{dubedat-coupling,ss-dgff,ss-contour}, see also \cite{werner-powell-gff, asw-btls}, and we will recall the relevant ideas now.

Let us consider a GFF $h$ with zero boundary conditions on $\partial \D$ and set $\lambda = \pi/2$. Moreover, fix $c\in (-\lambda,\lambda)$. Then one can deterministically associate a curve $\eta$ from $-i$ to $i$ to the field $h$ such that the conditional law of $h$ given $\eta$ is a GFF with zero boundary condition $\partial \D$, $-\lambda$ on the right side of $\eta$ and $\lambda$ on the left side of $\eta$ as illustrated in Figure \ref{fig:gff-coupling}. In fact, if $\eta'$ is another simple curve from $-i$ to $i$ coupled to $h$ with the same description of the conditional law of $h$ given $\eta'$ then $\eta = \eta'$ almost surely. It is also true that $\eta$ only depends \enquote{locally} on $h$ which is formalized by the notion of a local set but we will not require this here. The curve $\eta$ is called the level line of $h$ at level $c$ from $-i$ to $i$. As part of the proof of this statement, one shows that marginally $\eta\sim \SLE_4(c/\lambda-1,-c/\lambda-1)$.

Another construction due to Miller and Sheffield, see \cite[Section 4.3]{asw-btls}, is the coupling of $\CLE_4^0$ with $h$. It is possible to deterministically associate to $h$ a collection $\Gamma=\{(\gamma_i,\sigma_i)\colon i\ge 1\}$ of simple disjoint loops $\gamma_i$ in $\D$ together with signs $\sigma_i\in \{\pm 1\}$ such that conditionally on $\Gamma$ the field $h$ restricted to the interior of the loop $\gamma_i$ is a GFF with $2\lambda\sigma_i$ boundary conditions and that these restrictions are (conditionally) independent when $i\ge 1$ varies. Moreover, marginally $\Gamma\sim \CLE_4^0$ i.e. $\{\gamma_i\colon i\ge 1\}$ is a a $\CLE_4$ and for $j\ge 1$ the $\sigma_j\sim U(\{\pm 1\})$ are i.i.d. and independent of the collection of loops $\{\gamma_i\colon i\ge 1\}$.

As explained in \cite[Proposition 5.3]{cle-percolations}, $\eta$ does not intersect the interior of any loop in $\Gamma$. We can thus define a continuous curve $\gamma$ as follows. One traces $\eta$ and whenever $\eta$ intersects a loop $(\gamma',1)\in \Gamma$ one follows $\gamma'$ in clockwise direction and similarly, when $\eta$ intersects a loop $(\gamma',-1)\in \Gamma$ one follows this loop $\gamma'$ in counterclockwise direction i.e.\ we attach the loops in $\Gamma$ that intersect $\eta$ in chronological order to $\eta$ to obtain a curve $\gamma$ (the fact that one obtains a continuous curve is essentially a consequence of the local finiteness of CLE). 
The level line description for the GFF in the complement of $\eta$ gives actually a simple direct description of the conditional law of $\gamma$ given $\eta$ (this is explained in \cite{cle-percolations, asw-btls}). To emphasize the dependence on $c\in (-\lambda,\lambda)$ below, let us write $\eta_c$ and $\gamma_c$.

Theorem \ref{thm:main-result-msw} from \cite{cle-percolations} can then be reformulated as follows: \emph{We have
 $\gamma_c\sim \SLE_4^{\langle \mu \rangle}(-2)$ for some $\mu =M'(c)\in \R$ where $\rho=c/\lambda-1$. Furthermore, the mapping $M'$ is an increasing bijection from $(-\lambda, \lambda)$ to $\R$ such that $M'(-c) = - M'(c)$. In fact $M'(c)=M(c/\lambda -1)$.}  
Our Theorem \ref{thm:main-result} then complements this result by identifying this bijection as $M' (c) = \pi \tan (c)$. 

It is worth emphasizing that while the coupling of the labeled $\CLE_4$ and its exploration paths with a GFF will almost never be used explicitly in the present paper,  it is nevertheless instrumental for a number of results that we build upon, for instance in the derivation of Theorem \ref {thm:main-result-msw}. 

\subsection{Outline and comments}

In order to help the reader understand how the proof will work, let us mention the following simple observation about stable Lévy processes. If $S'$ is a stable Lévy process of exponent $1$ then necessarily, $S'=(aC_t+\mu t\colon t\ge 0)$ where $C$ is a Cauchy process, and $a \ge 0$ and $\mu\in \R$ are constants. On the other hand, a stable Lévy process of exponent $\alpha \in (1,2)$ is of the form $S=u_+S^+ - u_-S^-$ where $u_\pm\ge 0$ and $S^\pm$ are spectrally positive (compensated) Lévy processes of exponent $\alpha$. When $\alpha\downarrow 1$ and one considers suitable values of $u_\pm$ with $u_+/u_-$ tending to 1, then $S$ tends to $S'$ (in distribution). In other words, the deterministic drift asymmetry in $S'$ can be viewed as the limit of vanishing asymmetry in the jumps of $S$. 

The exploration paths in a $\CLE_\kappa$ for $\kappa\in (8/3,4]$ have a Loewner chain description that can be constructed from a Bessel process of dimension $\delta = 3 -8/\kappa$ and via the excursion decomposition of this Bessel process we are naturally led to considering stable Lévy processes of exponent $\alpha = 2-\delta=8/\kappa-1$ (this Lévy process will correspond to a certain compensated integral process, reparameterized by Bessel local time). The above result on stable Lévy process will then play a key role in approximating the $\CLE_4$ exploration paths by $\CLE_\kappa$ exploration paths; indeed, the parameter $\mu$ in the definition of $S$ is precisely the one appearing in the Loewner chain description of $\CLE_4$ exploration paths and the asymmetry parameter $u_+/u_-$ appears in the Loewner chain description of $\CLE_\kappa$ explorations for $\kappa<4$.

The core of the present paper will be devoted to the continuity result for the laws of $\CLE_\kappa$ exploration paths and their trunks when $\kappa$ varies. Section \ref{sec:classical} will be devoted to the case of the trunks and Section \ref {sec:gensle} will be devoted to the exploration paths themselves. The final short Section \ref{sec:mainresult} will put the pieces together and conclude the proof of Theorem \ref{thm:main-result}. 

It is worth stressing here that one key result in the present paper will be Theorem \ref{thm:bessel-int-conv} in Section \ref{sec:gensle}, which in essence states continuity in law when $\delta \uparrow 1$ of the joint law of a Bessel process of dimension $\delta$ and a compensated integral process associated to it (relevant definitions and results on Bessel processes appear in \cite{donati-yor-bessel,revuz-yor,yor-bm-eth}); this is a statement about Bessel processes only that turns out (as is often the case in the context of SLE curves) to have implications for two-dimensional random geometry.

Let us conclude this section by mentioning that the convergence statements of curves that we will derive and use in this paper are all in the sense of Carathéodory convergence (i.e. uniform convergence on compacts of the mapping out functions). Recently, analytic tools have been successfully deployed to prove stronger convergence results for (usual) $\SLE_\kappa$ processes. In \cite{friz-sle-continuity-grr} (strengthening the main result in \cite{sle-cont-viklund}) it was for instance established that if $B$ is a standard Brownian motion then a.s.\ for all $\kappa\in [0,8/3)$ the process $\sqrt{\kappa}B$ generates a continuous curve $\gamma_\kappa$ (parameterized by half-plane capacity) and that $(t,\kappa)\mapsto \gamma_\kappa(t)$ is jointly Hölder continuous on all compacts contained in $[0,\infty)\times (0,8/3)$ where the Hölder exponent and constant may depend on the compact subset. The sense of convergence required to complete the proof of our main theorem here is however much weaker than the stability results in the aforementioned papers. It seems however very tricky to generalize the arguments of \cite{friz-sle-continuity-grr,sle-cont-viklund} all the way to $\kappa = 4$ and even more difficult to see how to possibly handle with these techniques the generalized $\SLE_\kappa$ curves with force points (as introduced in Section \ref{subsec:sle-force-points} and Section \ref{sec:gensle}), since these generalized processes are quite delicate -- note for instance that our purpose is precisely to control the randomness that is `created' when the Bessel processes are equal to $0$.

\medspace

{\bf Acknowledgments.} The author was supported by grant 175505 of the Swiss National Science Foundation and is part of SwissMAP. He would like to thank Wendelin Werner for many insightful inputs throughout this project and the anonymous referee for helpful comments.

\section{SLEs with force points}
\label{sec:classical}

\subsection{Overview}

Let us first provide a brief heuristic overview of the strategy of our proof of  Theorem \ref{thm:main-result}, in order to provide some motivation for the coming sections. We will make extensive use of results from the two papers  \cite{cle-percolations} and \cite{msw-simple} about explorations of $\CLE_\kappa$ for $\kappa \in (8/3, 4)$.
For such values of $\kappa$ the explorations are parameterized by $\beta \in [-1,1]$, so that each $\CLE_\kappa$ loop is now oriented clockwise or anticlockwise independently, with respective probabilities $(1- \beta )/ 2$ and $(1+ \beta)/2$. The exploration curve $\gamma$ then traces the CLE loops that it encounters in its orientation. Just as the $\CLE_4$ exploration, the trunk $\eta$ of $\gamma$ will pass to the left (resp. the right) of the loops $\gamma_i$ with $\sigma_i = -1$ (resp. $\sigma_i = +1$) that it encounters. Let us insist on the fact that when $\beta \not= 0$ and $\kappa \in (8/3, 4)$, the labels of the $\CLE_\kappa$ loops are $\pm 1$ with probabilities different from $1/2$ (while for $\CLE_4$ explorations, one has to stick to the $\CLE_4^0$ labeled $\CLE_4$). 

In these papers it is shown for $\kappa\in (8/3,4)$ and $\beta\in [-1,1]$ that the trunk $\eta$ of such an 
$\SLE_\kappa^\beta(\kappa-6)$ exploration $\gamma$ is a $\SLE_{\kappa'}(\rho'(\kappa,\beta),\kappa'-6-\rho'(\kappa,\beta))$ curve for $\kappa' = 16/\kappa$ (so it is not a simple curve anymore), and the conditional law of $\gamma$ given its trunk is described in terms of what are called Boundary Conformal Loop Ensembles (BCLE).

Furthermore, using rather elaborate Liouville Quantum Gravity considerations, it is shown that the relation between $\rho'(\kappa,\beta)\in [\kappa'-6,0]$ and $\beta$ is 
\begin{align*}
	\tan(\pi\rho'(\kappa,\beta)/2)= \frac{\sin(\pi\kappa'/2)}{1+\cos(\pi\kappa'/2)-2/(1-\beta)}\;.
\end{align*}
Similarly, the trunk of $\gamma\sim \SLE_4^{\langle\mu\rangle}(-2)$ is a $\SLE_4(\rho'(4,\mu),-2-\rho'(4,\mu))$ curve for some value $\rho'(4,\mu)\in (-2,0)$ and the conditional law of $\gamma$ given its trunk can be described as well by attaching BCLE loops (as mentioned in the introduction). The main outcomes of the coming sections (i.e., respectively of Section \ref {sec:gensle} and Section \ref {sec:classical}) will be that 
\begin{align*}
	\SLE^{\mu(4/\kappa-1)}_\kappa(\kappa-6)&\to\SLE_4^{\langle \mu\rangle}(-2) \\
	\SLE_{\kappa'}(\rho'(\kappa,(4/\kappa-1) \mu),\kappa'-6-\rho'(\kappa,(4/\kappa-1) \mu)) &\to \SLE_4(\rho''(4,\mu),-2-\rho''(4,\mu))
\end{align*}
as $\kappa\uparrow 4$ in distribution with respect to uniform convergence on compacts of the driving functions where $\rho''(4,\mu)=\lim_{\kappa\uparrow 4} \rho'(\kappa,\mu(4/\kappa-1))$ is
\begin{align*}
	\rho''(4,\mu)= 2/\pi\cdot \arctan(\mu/\pi)-1
\end{align*}
and where $\arctan(\cdot)$ takes values in $(-\pi/2,\pi/2)$. 

We would like to deduce that $\rho'(4,\mu)=\rho''(4,\mu)$ thus completing the proof of the main theorem. The main difficulty is that the property of one curve being constructible from another curve by attaching random collections of BCLE loops does not pass easily to the limit. We will circumvent this by using one particular observable (the swallowing time of a point on the real axis by the CLE exploration path) the law of which characterizes the value $\rho'(4,\mu)$. This will be the content of Section \ref{sec:mainresult}.

\begin{remark}
	An alternative strategy to prove our main results would have been to build on the results from \cite{msw-non-simple} on asymmetric explorations of $\CLE_\kappa$'s for $\kappa > 4$, i.e., to take the limit $\kappa \downarrow 4$ instead of $\kappa \uparrow 4$. Yet another option would have been to try to directly derive results about $\CLE_\kappa$ decorations on LQG surfaces with parameter $\gamma=\sqrt{\kappa}$ for $\kappa =4$ and $\gamma = 2$, building on the LQG theory \cite{holden-powell-critical-lqg} in that case. But it seems that given the current literature, the approach chosen in the present paper is the shortest one to derive our results and we believe that understanding how to approximate $\CLE_4$ explorations by other $\CLE_\kappa$ explorations is also interesting on its own right. 
\end{remark}

\begin{remark} 
	The CLE exploration processes, when drawn on an appropriate LQG surface conjecturally correspond to the scaling limits of the discrete peeling processes on planar maps, and the formulas relating $\beta$ and $\rho$ also appear in the context of the study of the scaling limits of planar maps -- see the discussion and references in \cite{msw-simple, msw-non-simple}. The case $\kappa =4$ studied in the present paper (if coupled with the appropriate LQG) is similarly the continuum counterpart of the peeling processes studied in \cite{budd-curien-marzouk}. 
\end{remark}

\begin{remark}
	The approximation of the GFF by cable-graph GFFs as initiated by Lupu in \cite{lupu-lcriff} has been very useful in order to understand couplings of the GFF with $\CLE_4$ (see \cite{werner-wu-explorations,aru-lupu-sepulveda}) but it does not seem to be so amenable to the study of the question that we investigate in the present paper. 
\end{remark}

The main goal of this section wil now be to derive results about the continuity with respect to $(\kappa, \rho_\pm)$ of the laws of $\SLE_\kappa (\rho_-,\rho_+)$ processes. 

\subsection{Some results on Loewner chains}
\label{subsec:loewner-chains}

Let us very briefly recall the main construction of (chordal) Loewner chains (we refer the reader to \cite{werner-notes,lawler-book} for in depth introductions to the subject). Let $K\subset \overline{\H}$ be a chordal hull i.e. $K$ is compact, $\H\setminus K$ is simply connected and $K$ is the closure of $\H\cap K$. Then the capacity and half-plane capacity of $K$ are defined by
\begin{align*}
	\capacity(K) &= \lim_{y\to \infty} y\,\P_{iy}(\beta_{\tau_K} \in K)\;,\\
	\hcap(K) &= \lim_{y\to \infty} y\,\E_{iy}(\Im \beta_{\tau_K})\;.
\end{align*}
where $\beta$ is a complex Brownian motion and $\tau_K=\inf\{t\ge 0\colon \beta_t \notin \H\setminus K\}$. Then there exists a unique continuous bijection $g_K\colon \overline{\H}\setminus K\to \overline{\H}\setminus [a_-(K),a_+(K)]$ for some $a_-(K)\le a_+(K)$ such that $g_K$ restricted to $\H\setminus K$ is the unique conformal transformation from $\H\setminus K$ to $\H$ with $g_K(z)=z+\alpha_K/z+O(|z|^{-2}\,)$ as $|z|\to \infty$ and then necessarily $\alpha_K = \hcap(K)$. Moreover, $g_K$ extends holomorphically to a sufficiently small neighborhood of every $x\in \R\setminus K$. Finally, we remark that $\capacity(K)=(a_+(K)-a_-(K))/\pi$.

The main insight of Loewner theory is the encoding of families of chordal hulls by a continuous real-valued function: Whenever $L\in [0,\infty]$ and $\xi\colon [0,L)\to \R$ is a continuous function, we will associate to it its maximal solution to the Loewner equation
\begin{align*}
	\begin{cases}
		\dot{g}_t(z) = {2}/ ({g_t(z)-\xi_t}) &\colon t<\zeta_z\\
		g_0(z)=z
	\end{cases}
\end{align*}
for $z\in \overline{\H}\setminus\{\xi_0\}$ and define $K_t = \{z\in \overline{\H}\colon \zeta_z\le t\}$ for $t<L$. It turns out that $K_t$ is a chordal hull with $\hcap(K_t)=2t$ and $g_t=g_{K_t}$ for all $t> 0$. We say that the Loewner chain $(K_t)$ is generated by a continuous curve $\gamma\colon [0,L)\to \overline{\H}$ if $\H\setminus K_t$ is the infinite connected component of $\H\setminus \gamma([0,t])$ for all $t\in [0,L)$. In that case in fact $g_t(z)\to \xi_t$ when $\H\setminus K_t \ni z\to \gamma_t$.

Throughout, we will use the following notation: Whenever $\xi^*$ is some driving function where $*$ is some superscript, $g^*$, $\zeta^*$, $K^*$ and $\gamma^*$ will denote the corresponding objects in the chordal Loewner theory introduced above.

The only convergence result that will be needed is the following standard result in the theory of Loewner chains. Since the case with a boundary point is frequently not mentioned in the literature, we provide a self-contained proof here for the reader's convenience.

\begin{lemma}
	\label{lem:loewner-stuff}
	Let $\xi^n,\xi\colon [0,\infty)\to \R$ be continuous, start at $0$ and suppose that $\xi^n\to \xi$ uniformly on compacts as $n\to \infty$. Also consider $x\in \R\setminus \{0\}$. Then $\zeta_x \le \liminf_{n\to \infty} \zeta_x^n$ and $g^n(x)\to g(x)$ and $(g^n)'(x)\to g'(x)$ as $n\to \infty$ uniformly on compacts contained in $[0,\zeta_x)$.
\end{lemma}

\begin{proof}
	Without loss of generality, let us assume that $x>0$. Let $f(x)=g(x)-\xi$ and $f^n(x)=g^n(x)-\xi^n$ and fix $0\le T< \zeta_x\wedge \liminf_{n\to \infty} \zeta_x^n$. As a first step, we will prove uniform convergence of $f^n(x)$ to $f(x)$ as $n\to \infty$ on $[0,T]$. By Loewner's equation, for $t\le T$ and $n$ sufficiently large
	\begin{align*}
		|f^n_t(x)-f_t(x)| &\le |\xi^n_t - \xi_t| + \int_0^t \left| \frac{2}{f^n_s(x)} - \frac{2}{f_s(x)} \right| \,ds \\
		&\le \sup_{[0,T]} |\xi^n - \xi| + \frac{2}{\inf_{[0,t]}|f^n(x)|\cdot \inf_{[0,t]}|f(x)|} \int_0^t |f^n_s(x)-f_s(x)|\,ds\;.
	\end{align*}
	Let $\delta_T = \inf_{[0,T]} |f(x)|>0$, consider $\epsilon \in (0,\delta_T/2)$ and define $\tau_n = \inf\{t\le T\colon |f^n_t(x)-f_t(x)| > \epsilon \}$ with the usual convention that $\tau_n=\infty$ if we are taking the infimum of an empty set. If $\tau_n < \infty$, then we get
	\begin{align*}
		|f^n_t(x)-f_t(x)| \le \sup_{[0,T]} |\xi^n - \xi| + (2/\delta_T)^2 \int_0^t |f^n_s(x)-f_s(x)|\,ds \quad\text{for $t\le \tau_n$}
	\end{align*}
	and by Grönwall's inequality therefore
	\begin{align*}
		|f^n_t(x)-f_t(x)| \le \sup_{[0,T]} |\xi^n - \xi| \cdot e^{(2/\delta_T)^2 T}\quad\text{for $t\le \tau_n$}\;.
	\end{align*}
	When $t=\tau_n$, the left-hand side equals $\epsilon$ so that we get a contradiction when $n$ is sufficiently large. Thus $\tau_n =\infty$ for $n$ sufficiently large as required.

	Let us now suppose for a contradiction that $\liminf_{n\to \infty} \zeta^n_x < \zeta_x$ and take $(n_k)$ with $\tau = \lim_{k\to \infty} \zeta^{n_k}_x < \zeta_x$. Then
	\begin{align*}
		g_\tau(x) &= \int_0^\tau \frac{2}{f_s(x)}\,ds = \int_0^\infty \liminf_{k\to \infty}\frac{2\cdot 1(s<\zeta^{n_k}_x)}{f^{n_k}_s(x)}\,ds \le \liminf_{k\to \infty} \int_0^{\zeta^{n_k}_x} \frac{2}{f^{n_k}_s(x)}\,ds\\
		&= \liminf_{k\to \infty} \xi^{n_k}_{\zeta^{n_k}_x} = \xi_\tau
	\end{align*}
	which is clearly a contradiction since we know that $f_\tau(x)> 0$ (as $\tau<\zeta_x$). To obtain the convergence of the derivatives, note that for $t<\zeta_x$,
	\begin{align*}
		\frac{d}{dt}\log g_t'(x) = \frac{\dot{g}_t'(x)}{g_t'(x)} = \frac{-2}{(g_t(x)-\xi_t)^2} \quad\text{and hence}\quad g_t'(x)=\exp ( {-2\int_0^t (g_s(x)-\xi_s)^{-2}\,ds}) 
	\end{align*}
	and the analogous statement holds for $(g^n)'(x)$. The result then immediately follows.
\end{proof}

We will also need the following fact that describes how the half-plane capacity changes when one changes the target point (this type of result was already used in the proof of target invariance of $\SLE_6$ in \cite{lsw-bm-exponents1}). Again we provide a quick self-contained proof.

\begin{lemma}
	\label{lem:change-hcap}
	Let $\xi\colon [0,\infty)\to \R$ with $\xi_0=0$ be continuous such that $\zeta_{-y} = \infty$ where $-y<0$. Let $\psi\colon \H\to \H$ be the Möbius transformation mapping $(0,\infty,-y)$ to $(0,x,\infty)$ with $x>0$. Let $\tilde{K}_t$ be the chordal hull associated to $\psi(K_t)$ for $t\ge 0$. Then
	\begin{align*}
		\hcap(\tilde{K}_t) = 2\int_0^t \left(\frac{xy\cdot g_t'(-y)}{(g_s(-y)-\xi_s)^2}\right)^2 \,ds\quad\text{for all $t\ge 0$}\;.
	\end{align*}
\end{lemma}

\begin{proof}
	We have $\psi(z)=xz/(y+z)$ for $z\in \H$. Let $\tilde{g}_t=g_{\tilde{K}_t}$ so that $h_t = \tilde{g}_t\circ \psi\circ g_t^{-1}\colon \H\to \H$ is a conformal transformation and hence a Möbius transformation. Since $g_t(z)-z\to 0$ as $|z|\to \infty$ and $\tilde{g}_t(w)-w\to 0$ as $|w|\to \infty$ it follows that
	\begin{align*}
		h_t(z)= \frac{\lambda_t}{g_t(-y)-z}+\mu_t \quad\text{for some constants $\lambda_t>0$ and $\mu_t\in \R$}\;.
	\end{align*}
	Note that $\mu_t = \tilde{g}_t(x)$. By \cite[Section 4.1]{werner-notes} we get for $t\ge 0$ that
	\begin{align*}
		\hcap(\tilde{K}_t) = 2\int_0^t h_s'(\xi_s)^2\,ds = 2\int_0^t \left(\frac{\lambda_s}{(g_s(-y)-\xi_s)^2}\right)^2 \,ds\;.
	\end{align*}
	To obtain $\lambda_t$ we write
	\begin{align*}
		0&=h_t(g_t(-y+i\epsilon)) - \tilde{g}_t(\psi(-y+i\epsilon)) = h_t(g_t(-y)+i\epsilon g'_t(-y) + O(\epsilon^2)) - \tilde{g}_t(x+ixy/\epsilon)\\
		&= i\lambda_t/(\epsilon g'_t(-y)) - ixy/\epsilon + O(1)\quad\text{as $\epsilon\to 0$}\;.
	\end{align*}
	Hence $\lambda_t = xy\cdot g_t'(-y)$ for $t\ge 0$ as required.
\end{proof}

\subsection{Parameter-dependence of classical SLEs with force points}
\label{subsec:sle-force-points}

The starting point of Schramm's seminal work \cite{schramm0} was that the Loewner chain generated by $\sqrt{\kappa}$ times a standard Brownian motion is the  only one with a conformal Markov property, which is expected to hold for the scaling limit of a large number of statistical physics models with appropriately chosen boundary conditions. Rohde and Schramm in \cite{schramm-sle} subsequently showed that this Loewner chain is almost surely generated by a continuous curve and the law of the resulting curve is called $\SLE_\kappa$. A first very natural variant of $\SLE_\kappa$ curves are $\SLE_\kappa(\rho_-, \rho_+)$ curves which appear naturally in a variety of settings when there is one additional marked boundary point. They were first introduced in the case when $\rho_-=0$ or $\rho_+=0$, and $\rho_\pm > -2$ in \cite{lsw-restriction}. In the works \cite{sle-martingales-dubedat,ss-dgff,ig1} they were more generally introduced under only the assumption that $\rho_\pm > -2$ (and also for more than two force points); we will refer to these processes as classical $\SLE_\kappa (\rho)$ type processes.

A  subtle generalization first proposed by Sheffield in \cite{shef-cle} allows  to also consider some values of $\rho$ smaller than $-2$. We will describe and discuss it later in Section \ref {sec:gensle}. 

Consider a standard Brownian motion $B$ and let $\kappa>0$, $\rho_-,\rho_+\in \R$. We now consider the following system of SDEs
\begin{align}
	\label{eq:slerhosde}
	\begin{cases}
		\xi_t = \xi_0 + \sqrt{\kappa} B_t + \int_0^t {\rho_-\,ds} /({\xi_s - O_s^-}) + \int_0^t {\rho_+\,ds}/({\xi_s-O_s^+})\;,\\
		O^\pm_t = O^\pm_0+\int_0^t {2\,ds}/({O_s^\pm - \xi_s})\;,\\
		O^-\le \xi \le O^+\;.
	\end{cases}
\end{align}
$O_0^\pm$ are called force points since on a heuristic level the Loewner chain generated by $\xi$ is attracted to (resp. repelled from) $O_0^\pm$ when $\rho_\pm < 0$ (resp. $\rho_\pm > 0$). Note that when $O^-_0 < \xi_0 < O^+_0$, we can solve the SDE up to the first time $\zeta$ (called the swallowing time) at which $O^+ - \xi$ or $O^- - \xi$ hit $0$. Since $\SLE_\kappa$ is generated by a continuous curve a.s., Girsanov's theorem directly implies that $(K_t\colon t<\zeta)$ is also generated by a continuous curve a.s.\ and the law of this curve is called an $\SLE_\kappa(\rho_-,\rho_+)$ started at $\xi_0$ with force points at $O^\pm_0$ up to the swallowing time.

As explained in \cite[Section 4]{ss-dgff} and also in \cite[Theorem 2.2]{ig1}, when $\rho^\pm > -2$ and for any initial conditions $O^-\le \xi_0 \le O^+_0$, there exists a unique solution in law defined on all of $[0,\infty)$ to the SDE \eqref{eq:slerhosde} and the triple $(\xi,O^\pm)$ forms a strong Markov process. It is moreover established in \cite[Theorem 1.3]{ig1} that the Loewner chain associated to the driving function $\xi$ is almost surely generated by a continuous curve and we call the law of this curve $\SLE_\kappa(\rho_-,\rho_+)$ started at $\xi_0$ with force points at $O^\pm_0$. When we talk about $\SLE_\kappa(\rho_-,\rho_+)$ without mentioning the parameters $(\xi_0,O^\pm_0)$ it is understood that they are zero. The scale invariance of the latter process allows us to define it in general domains with two distinguished prime ends via a conformal mapping.

Of course, one expects that the law of the process $(\xi,O^\pm)$ is in some sense continuous in the parameters $\kappa$ and $\rho_\pm>-2$ for given initial conditions -- the first goal of the present section will be actually to prove the following statement (we will then derive some slightly strengthened variants of this result):

\begin{prop}
	\label{prop:sle-rho-convergence}
	Suppose that $\rho^n_\pm>-2$ and $\kappa^n>0$ for $1\le n\le \infty$ be such that $(\kappa^n,\rho^n_\pm)\to (\kappa^\infty,\rho^\infty_\pm)$ as $n\to \infty$. Let $(\xi^n,O^{n\pm})$ solve the SDE \eqref{eq:slerhosde} with $(\kappa,\rho^\pm)$ replaced by $(\kappa^n,\rho^{n\pm})$ and with $\xi^n_0=O^{n\pm}_0=0$ for all $1\le n\le \infty$. Then $(\xi^n,O^{n\pm})\to (\xi^\infty,O^{\infty\pm})$ as $n\to \infty$ in distribution with respect to uniform convergence on compact sets.
\end{prop}

The singularities in the SDE \eqref{eq:slerhosde} make it a little tricky to adapt standard results in the literature. There are various possible ways to prove this; we opt here for a hands-on self-contained approach based on the change of coordinates given in \cite[Section 4]{ss-dgff}, and we refer the reader to the same place for the geometric significance of this approach.
 
Before proceeding, let us first briefly write down some standard results from the theory of Imaginary Geometry \cite{ig1} which are also recalled in \cite[Section 8]{cle-percolations} (for the $\kappa=4$ case, consider the level line coupling with the GFF) and a straightforward consequence which will be important later in our paper. See also Figure \ref{fig:ig-lemmas} where the results below are illustrated.

\begin{lemma}[\cite{ig1}]
	\label{lem:ig-angles}
	Let $\kappa\in (0,4]$ and consider $\rho_+>-2$ and $\rho_-\in (-2,0)$. Then one can couple $\eta_+\sim \SLE_\kappa(0,\rho_+)$ with $\eta_- \sim \SLE_\kappa(-2-\rho_-,2+\rho_-+\rho_+)$ (both from $0$ to $\infty$ in $\H$) such that $\eta_-$ lies left of $\eta_+$ and conditionally on $\eta_-$, the restrictions of $\eta_+$ to the components right of $\eta_-$ are independent $\SLE_\kappa(\rho_-,\rho_+)$ curves.
\end{lemma}

\begin{lemma}[\cite{ig1}]
	\label{lem:ig-right-boundary}
	Consider $\kappa\in (0,4]$, $\kappa'=16/\kappa$, $\rho_+'>-2$ and $\rho_-'>\kappa'/2-4$. Let $\eta'$ be an $\SLE_{\kappa'}(\rho'_-,\rho'_+)$ from $i$ to $0$ in the domain $\R+i(0,1)$. Let $\eta_R$ be the right outer boundary of $\eta'$ which is a continuous curve from $0$ to $i$. Then
	\begin{align*}
		\eta_R &\sim \SLE_\kappa(\kappa/4\cdot(\rho'_++2)-2,\kappa/4\cdot(\rho'_-+4)-4)
	\end{align*}
	from $0$ to $i$ in the domain $\R+i(0,1)$.
\end{lemma}

\begin{lemma}
	\label{lem:ig-main-iteration}
	Consider $\kappa\in (0,4]$, $\kappa'=16/\kappa$, $\rho'_+\in (-2,\kappa'/2-2)$, and let $\rho=-\kappa/4\cdot(\rho'_++2)$. Then one can couple $\eta'\sim \SLE_{\kappa'}(0,\rho'_+)$ from $i$ to $0$ in $\R+i(0,1)$ with $\eta\sim \SLE_\kappa(0,\kappa-6-\rho)$ from $0$ to $i$ in $\R+i(0,1)$ such that $\eta$ lies right of $\eta'$ and conditionally on $\eta'$, $\eta$ restricted to the components right of $\eta'$ is given by independent $\SLE_\kappa(\rho,\kappa-6-\rho)$ curves.
\end{lemma}

\begin{proof}[Proof of Lemma \ref{lem:ig-main-iteration}] 
	Let $\eta_R$ be the right boundary of $\eta'$ (a curve from $0$ to $i$). Then by Lemma \ref{lem:ig-right-boundary}, we have $\eta_R\sim \SLE_{\kappa'}(\kappa/4\cdot (\rho'_+ + 2)-2,\kappa-4)$. The lemma now follows by the coupling of $\eta_R$ and $\eta$ as in Lemma \ref{lem:ig-angles} (with $\rho_-=\rho$ and $\rho_+=\kappa-6-\rho$) and indeed, we can consider a coupling such that $\eta'$ and $\eta$ are conditionally independent given $\eta_R$. The claim follows.
\end{proof}

\begin{figure}
	\centering
	\def\svgwidth{0.9\columnwidth}
	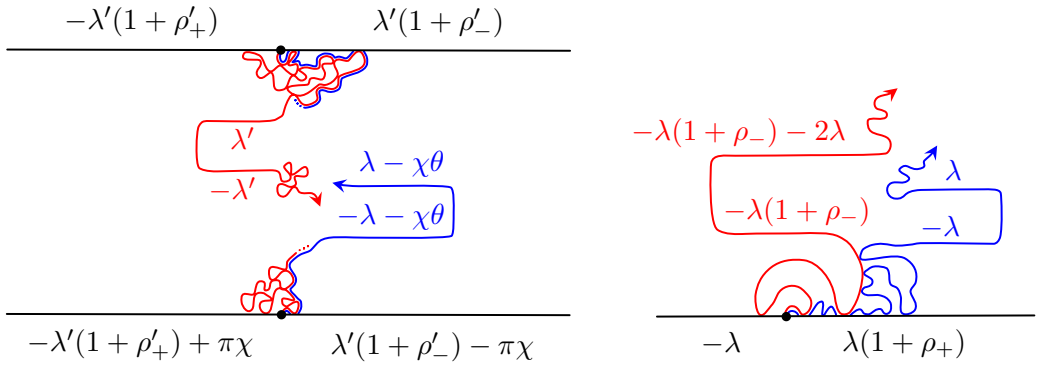
	\caption{\emph{Left.} This figure illustrates Lemma \ref{lem:ig-right-boundary} with $\eta$ being the red curve and the blue simple one is its right boundary $\eta_R$. When $\kappa=4$ the result is trivial, otherwise the boundary conditions appearing in the imaginary geometry proofs have been drawn in where
	$\lambda=\pi/\sqrt{\kappa}$, $\lambda'=\pi/\sqrt{\kappa'}$, $\chi=2/\sqrt{\kappa}-\sqrt{\kappa}/2$  and $\theta=-\pi/2$. 
	\emph{Right.} This is an illustration of Lemma \ref{lem:ig-angles} where the blue curve depicts $\eta_+$ and the red curve $\eta_-$. Again, the imaginary geometry boundary conditions have been drawn in where $\lambda=\pi/\sqrt{\kappa}$.}
	\label{fig:ig-lemmas}
\end{figure}

Let us now go back to our road towards the proof of Proposition \ref{prop:sle-rho-convergence}. Suppose that $O^-_0<O^+_0$. From the SDE \eqref{eq:slerhosde} it is clear that $O^+$ is increasing and $O^-$ is decreasing. It is also easy to see that $O^+_t-O^-_t\to \infty$ as $t\to\infty$ a.s.\ and we may work with a version where this holds surely. Let us define
\begin{align}
	\label{eq:changecoord}
	\begin{split}
		\sigma_t &= \log( O_t^+ - O_t^-) -  \log( O_0^+ - O_0^-)\;,\\
		\tau_s &= \inf\{t\ge 0\colon \sigma_t \ge s\}\;,\\
		Y_s &= \frac{2\xi_{\tau_s}-(O_{\tau_s}^+ + O_{\tau_s}^-)}{O^+_{\tau_s} - O^-_{\tau_s}} = e^{-s}\,\frac{2\xi_{\tau_s}-(O_{\tau_s}^+ + O_{\tau_s}^-)}{O^+_0 - O^-_0} \;.
	\end{split}
\end{align}
From the above remarks, we see that all three processes are continuous and the monotone bijections $\sigma,\tau\colon [0,\infty)\to [0,\infty)$ are inverses of each other. Before proceeding, let us invert theses relations: By \eqref{eq:changecoord} we see that
\begin{align}
	\label{eq:inversefirst}
	\begin{split}
		O_t^\pm - \xi_t &= \frac{O_0^+-O^-_0}{2}\,e^{\sigma_t}(\pm 1-Y_{\sigma_t})\;,\\
		\dot{\tau}_s &= \frac{1}{\dot{\sigma}_{\tau_s}} = \frac{O^+_{\tau_s}- O^-_{\tau_s}}{2/(O^+_{\tau_s}-\xi_{\tau_s})-2/(O^-_{\tau_s}-\xi_{\tau_s})} = \frac{1}{2}\,(O^+_{\tau_s}-\xi_{\tau_s})(\xi_{\tau_s}-O^-_{\tau_s})\\
		&= \frac{(O^+_0 - O^-_0)^2}{8}\,e^{2s} (1-Y_s^2)\;.
	\end{split}
\end{align}
This allows us to determine $\tau$ and hence its inverse $\sigma$ from $Y$. We have thus obtained an explicit formula for $O^\pm-\xi$ in terms of $Y$. From \eqref{eq:inversefirst} and an elementary computation, we get
\begin{align}
	\label{eq:inversesecond}
	\begin{split}
		\xi_t &= O^+_t - (O^+_t - \xi_t) = O_0^+ + \int_0^t \frac{2}{O^+_u-\xi_u}\,du - (O^+_t - \xi_t)\\
		&= O^+_0 + \int_0^{\sigma_t} \frac{2\dot{\tau}_r}{O^+_{\tau_r}-\xi_{\tau_r}}\,dr - (O^+_t - \xi_t)\\
		&= \frac{O_0^+ + O_0^-}{2} + \frac{O_0^+ - O_0^-}{2}\left( e^{\sigma_t}Y_{\sigma_t} + \int_0^{\sigma_t} e^r Y_r\,dr \right)\;.
	\end{split}
\end{align}
Next, let us see which SDE is satisfied by $Y$. In the following computation, we first use the definition of $Y$ given in \eqref{eq:changecoord}, then the given SDE \eqref{eq:slerhosde} and finally we make use of the inverted relations \eqref{eq:inversefirst} to obtain the final expression:
\begin{align}
	\label{eq:besselsquaredapprox}
	\begin{split}
		dY_s &= - Y_s\,ds + \frac{e^{-s}}{O^+_0 - O^-_0}\,\left( 2\cdot d(\xi\circ \tau)_s - d(O^+\circ \tau)_s - d(O^-\circ \tau)_s \right)\\
		&= - Y_s\,ds + \frac{e^{-s}}{O^+_0 - O^-_0}\,\left( 2\sqrt{\kappa}\,d(B\circ \tau)_s - \frac{2\rho_+ +2}{O^+_{\tau_s}-\xi_{\tau_s}}\,\dot{\tau}_s\,ds - \frac{2\rho_- +2}{O^-_{\tau_s}-\xi_{\tau_s}}\,\dot{\tau}_s\,ds \right)\\
		&= -\frac{\rho_+ +2}{2}(Y_s+1)\,ds - \frac{\rho_- + 2}{2}(Y_s-1)\,ds + \sqrt{\frac{\kappa}{2}\,(1-Y_s^2)}\cdot \dot{\tau}_s^{-1/2}\,d(B\circ \tau)_s\\
		&= -\frac{\rho_+ +2}{2}(Y_s+1)\,ds - \frac{\rho_- + 2}{2}(Y_s-1)\,ds + \sqrt{\frac{\kappa}{2}\,(1-Y_s^2)}\,d\tilde{B}_s
	\end{split}
\end{align}
where the process $\tilde{B}=\int_0^{\,\cdot} \dot{\tau}_s^{-1/2}\,d(B\circ \tau)_s$ is a standard Brownian motion. We will now analyze the resulting SDE in some detail. First, the generator $\mathcal{L}_{\kappa,\rho_\pm}$ of the process $Y$ is
\begin{align*}
	\mathcal{L}_{\kappa,\rho_\pm}f(y) &= \left(-\frac{\rho_+ + 2}{2}(y+1)-\frac{\rho_- +2}{2}(y-1)\right)f'(y)+ \frac{\kappa}{4}(1-y^2)f''(y)
\end{align*}
defined for all twice differentiable $f$. The reader can then check that the invariant distribution is given by
\begin{align*}
	\mu_{\kappa,\rho_\pm}(dy) = \frac{1_{(-1,1)}(y)/2\,dy}{B(2(\rho_-+2)/\kappa,2(\rho_++2)/\kappa)} \left(\frac{1+y}{2}\right)^{2(\rho_-+2)/\kappa-1} \left(\frac{1-y}{2}\right)^{2(\rho_+ +2)/\kappa-1}
\end{align*}
where in the denominator $B$ denotes the Beta function. By a coupling argument, the law of $Y_s$ converges to $\mu_{\kappa,\rho_\pm}$ as $s\to \infty$. From this, one can deduce the following lemma.

\begin{lemma}
	\label{lem:inv-dist-sle-rho}
	Let $(\xi,O^\pm)$ solve \eqref{eq:slerhosde} with $\xi_0=O_0^\pm=0$. Let $\eta_\epsilon=\inf\{t\ge 0\colon O^+_t-O^+_t=\epsilon\}$. Then for any $\epsilon>0$,
	\begin{align*}
		\frac{2\xi_{\eta_\epsilon}-(O^+_{\eta_\epsilon}+O^-_{\eta_\epsilon})}{\epsilon}\sim \mu_{\kappa,\rho_\pm}\;.
	\end{align*}
\end{lemma}

The idea of the proof is to consider the conditional law of $(2\xi_{\eta_\epsilon}-(O^+_{\eta_\epsilon}+O^-_{\eta_\epsilon}))/\epsilon$ given $\xi|_{[0,\eta_{\epsilon'}]},O^\pm_{[0,\eta_{\epsilon'}]}$ for $\epsilon'\in (0,\epsilon)$ and consider the limit $\epsilon'\to 0$. We leave the details to the reader.

\begin{prop}
	\label{prop:approxslekapparho}
	Let $\rho_\pm^n\in \R$, $\kappa^n>0$ and $y^n_0\in [-1,1]$ for $1\le n\le \infty$ be such that $y^n_0\to y_0^\infty$ and $(\kappa^n,\rho_\pm^n) \to (\kappa^\infty,\rho_\pm^\infty)$ as $n\to \infty$. Moreover, for all $1\le n\le \infty$, let $Y^n$ be the unique (strong) solution with initial condition $Y^n_0 = y^n_0$ to the SDE 
	\begin{align*}
		dY^n_s = -\frac{\rho^n_+ +2}{2}(Y_s+1)\,ds - \frac{\rho^n_- + 2}{2}(Y_s-1)\,ds + \sqrt{\frac{\kappa^n}{2}\,(1-Y_s^2)}\,d\tilde{B}_s\;.
	\end{align*}
	Then $Y^n \to Y^\infty$ as $n\to\infty$ in distribution with respect to uniform convergence on compact sets.
\end{prop}

The fact that solutions to the above SDEs with values in $[-1,1]$ exist, that they are strong and that pathwise uniqueness holds are consequences of classical SDE theory and we refer the reader to \cite[Chapter IX]{revuz-yor}.

\begin{proof}
	We follow the classical strategy of first establishing tightness and then using Stroock-Varadhan theory to conclude. From the SDE and the BDG inequality we see that
	\begin{align*}
		\E\left((Y^n_r - Y^n_s)^{2p}\right) \lesssim_{\,p,(\kappa^m),(\rho^m_\pm)} |r-s|^{2p} + |r-s|^p
	\end{align*}
	for all $r,s\ge 0$, $1\le n\le \infty$ and $p\ge 1$. A dyadic subdivision argument permits us to bound the expectation of the Hölder norm of $Y^n$ with exponent $\alpha$ on $[0,S]$ uniformly in $1\le n\le \infty$ for each fixed $S\ge 0$ and $\alpha \in (0,1/2)$; combining this with the boundedness of $(y^n_0)$ then yields tightness of the family of processes $(Y^n)$ in the topology of uniform convergence on compact sets -- we refer the reader to \cite[Corollary 16.9]{kallenberg} (this is a quantitative version of Kolmogorov's theorem) for details. By Prokhorov's theorem and Skorokhod's representation theorem, for each subsequence $(n_k)$, there exists a further subsequence $(n_{k_l})$ and random processes $Z^l$ for $1\le l\le \infty$ such that $Y^{n_{k_l}}=_d Z^l$ for $l<\infty$ and $Z^l\to Z^\infty$ uniformly on compact sets a.s.\ as $l\to \infty$. Now by Itô's formula, for each $1\le n\le \infty$ and $f\in C^\infty(\R)$,
	\begin{align*}
		M^{f,n} = f(Y^n) - f(y^n_0) + \int_0^{\,\cdot} &\Bigl( \frac{\rho^n_++2}{2}(Y^n_r + 1)f'(Y^n_r) + \frac{\rho^n_-+2}{2}(Y^n_r - 1)f'(Y^n_r) \\
		&\qquad- \frac{\kappa^n(1-(Y^n_r)^2)}{4}\,f''(Y^n_r) \Bigr) \,dr
	\end{align*}
	is a martingale starting from $0$ in the filtration generated by itself. Since $Y^{n_{k_l}}=_d Z^l$ for $l<\infty$ and $Z^l\to Z^\infty$ uniformly on compact sets a.s.\ as $l\to \infty$, we deduce from the fact that $(\kappa^n,\rho^n_\pm)\to (\kappa^\infty,\rho^\infty_\pm)$ and $y^n_0\to y^\infty_0$ as $n\to \infty$ that
	\begin{align*}
		f(Z^\infty) - f(y^\infty_0) + \int_0^{\,\cdot} &\Bigl( \frac{\rho^\infty_++2}{2}(Z^\infty_r + 1)f'(Z^\infty_r) + \frac{\rho^\infty_-+2}{2}(Z^\infty_r - 1)f'(Z^\infty_r) \\
		&\qquad- \frac{\kappa^\infty(1-(Z^\infty_r)^2)}{4}\,f''(Z^\infty_r) \Bigr) \,dr
	\end{align*}
	is also a martingale starting from $0$ in its own filtration. By Stroock-Varadhan theory, see for instance \cite[Theorem 21.7]{kallenberg}, this implies that $Z^\infty=_d Y^\infty$; note that here we are using the uniqueness in law property of the given SDE. Since this is true for any subsequence $(n_k)$, the claim readily follows.
\end{proof}

\begin{cor}
	\label{cor:sle-rho-convergence-nonzero}
	Suppose that $\rho^n_\pm>-2$, $\kappa^n>0$ and $x^n_-<x^n_+$ with $x^n_-\le 0\le x^n_+$ for $1\le n\le \infty$ be such that $(\kappa^n,\rho^n_\pm,x^n_\pm)\to (\kappa^\infty,\rho^\infty_\pm,x^\infty_\pm)$ as $n\to \infty$. Let $(\xi^n,O^{n\pm})$ solve the SDE \eqref{eq:slerhosde} with $(\kappa,\rho^\pm)$ replaced by $(\kappa^n,\rho^{n\pm})$ and with $\xi^n_0=0$, $O^{n\pm}_0=x^n_\pm$ for all $1\le n\le \infty$. Then $(\xi^n,O^{n\pm})\to (\xi^\infty,O^{\infty\pm})$ as $n\to \infty$ in distribution with respect to uniform convergence on compact sets.
\end{cor}

\begin{proof}
	Let $(Y^n)$ be as in Proposition \ref{prop:approxslekapparho} with $y^n_0 = -(x^n_+ + x^n_-)/(x^n_+ - x^n_-)$ and the same $(\kappa^n,\rho^n_\pm)$ values for $n\le \infty$. By Skorokhod’s representation theorem and Proposition \ref{prop:approxslekapparho}, there are processes $Z^n$ such that $Z^n=_d Y^n$ for all $n\le \infty$ and such that $Z^n\to Z^\infty$ uniformly on compacts almost surely as $n\to \infty$. By \eqref{eq:changecoord}, \eqref{eq:inversefirst} and \eqref{eq:inversesecond} it suffices to show that $\tilde{\sigma}^n\to \tilde{\sigma}^\infty$ uniformly on compacts a.s.\ as $n\to \infty$ where
	\begin{align*}
		\tilde{\tau}^n = \frac{(x^n_+-x^n_-)^2}{8}\int_0^{\,\cdot} e^{2r} (1-(Z^n_r)^2)\,dr\quad\text{and}\quad \tilde{\sigma}^n_t = \inf\{s\ge 0\colon \tilde{\tau}_s\ge t\}\quad\text{for $t\ge 0$}\;.
	\end{align*}
	Clearly $\tilde{\tau}^n\to \tilde{\tau}^\infty$ uniformly on compacts a.s.\ as $n\to \infty$. The result follows since $\tilde{\tau}^\infty$ is strictly increasing a.s.
\end{proof}

\begin{proof}[Proof of Proposition \ref {prop:sle-rho-convergence}]
	We can finally proceed with this proof: We fix $\epsilon>0$ and define stopping times $\eta_\epsilon^n=\inf\{t\ge 0\colon O^{n+}_t-O^{n-}_t = \epsilon\}$. Then by Lemma \ref{lem:inv-dist-sle-rho} we have
	\begin{align*}
		\frac{2\xi_{\eta_\epsilon^n}^n-(O^{n+}_{\eta_\epsilon^n}+O^{n-}_{\eta_\epsilon^n})}{\epsilon}\sim \mu_{\kappa^n,\rho^n_\pm}\quad\text{for $n\le \infty$ and $\epsilon>0$}\;.
	\end{align*}
	Since $\mu_{\kappa^n,\rho^n_\pm}\to \mu_{\kappa^\infty,\rho^\infty_\pm}$ weakly as $n\to \infty$, Corollary \ref{cor:sle-rho-convergence-nonzero} and the strong Markov property of $(\xi^n,O^{n\pm})$ at time $\eta^n_\epsilon$ imply that
	\begin{align*}
		(\xi^n_{\cdot+\eta^n_\epsilon}-\xi^n_{\eta^n_\epsilon},O^{n+}_{\cdot+\eta^n_\epsilon}-\xi^n_{\eta^n_\epsilon},O^{n-}_{\cdot+\eta^n_\epsilon}-\xi^n_{\eta^n_\epsilon}) \to (\xi^\infty_{\cdot+\eta^\infty_\epsilon}-\xi^\infty_{\eta^\infty_\epsilon},O^{\infty+}_{\cdot+\eta^\infty_\epsilon}-\xi^\infty_{\eta^\infty_\epsilon},O^{\infty-}_{\cdot+\eta^\infty_\epsilon}-\xi^\infty_{\eta^\infty_\epsilon})
	\end{align*}
	in distribution with respect to uniform convergence on compacts as $n\to \infty$ for all $\epsilon>0$. Since $\pm O^{n\pm}$ are non-decreasing and $O^{n-}\le \xi^n\le O^{n+}$ we get that
	\begin{gather*}
		\sup_{[0,\eta^n_\epsilon]} (|\xi^n| + |O^{n+}| + |O^{n-}| ) \lesssim \epsilon\;,\\
		\epsilon \ge \int_0^{\eta^n_\epsilon} \frac{2}{O^{n+}_t - \xi^n_t}\,dt \ge 2\eta^n_\epsilon/\epsilon \quad\text{so that}\quad \eta^n_\epsilon \le \epsilon^2/2\;.
	\end{gather*}
	Note that the last bound above can be seen as bounding the half-plane capacity by the capacity squared (divided by two). Furthermore
	\begin{align*}
		\sup_{r\in [0,\epsilon^2/2]}\sup_{[0,T]}\left(|\xi^\infty_{\cdot+r}-\xi^\infty|+|O^{\infty+}_{\cdot+r}-O^{\infty+}|+|O^{\infty-}_{\cdot+r}-O^{\infty-}|\right) \to 0
	\end{align*}
	as $\epsilon \to 0$ for all $T\ge 0$. The claim follows by combining the observations above (using for instance Skorokhod’s representation theorem to organize the argument).
\end{proof}

The remaining results in this section will now build on Proposition \ref{prop:sle-rho-convergence} to derive the results that will actually be used at the end of the paper. The main additional feature is related to the fact that we also want to control the swallowing times of given points. It is worth recalling here that an $\SLE_\kappa(\rho_-,\rho_+)$ with $\rho_-, \rho_+ > -2$ will almost surely hit the positive half-line as soon as $\rho_+ < -2 + \kappa / 2$, i.e., all swallowing times $\zeta_x$ for $x>0$ will almost surely be finite.

\begin{prop}
	\label{prop:boosted-sle-conv}
	Suppose that $(\kappa^n,\rho^n_\pm,\xi^n,O^{n\pm})$ are as in Proposition \ref{prop:sle-rho-convergence}. Assume also that $\rho^n_+ < \kappa^n/2 - 2$ and $\rho^n_+\le 0$ for all $n\le \infty$. Then for $x>0$,
	\begin{align*}
		(\xi^n,O^{n\pm},\zeta^n_x) \to (\xi^\infty,O^{\infty\pm},\zeta^\infty_x)
	\end{align*}
	in distribution as $n\to \infty$ with respect to uniform convergence on compacts in the first three components and with respect to the Euclidean metric in the final component of the tuple.
\end{prop}

The proof of this proposition will build on the following lemma: 

\begin{lemma}
	\label{lem:hit-boundary-fast}
	Fix $\epsilon>0$. Then for all $\delta>0$ there exists $C_\delta, t_\delta >0$ such that the following is true. Suppose $(\kappa,\rho^\pm)$ are such that $\epsilon\le \kappa\le 1/\epsilon$, $\rho_-\le 1/\epsilon$ and $\rho_+\le (\kappa/2-2-\epsilon\kappa)\wedge 0$. Moreover let $(\xi,O^{\pm})$ solve the SDE \eqref{eq:slerhosde} with $\xi_0=0$, $O^-_0\le -C_\delta$ and $O^+_0\in [0,1]$. Then $\P(\zeta_1\ge t_\delta)\le \delta$.
\end{lemma}

\begin{proof}
	Let $(\tilde{\xi},\tilde{O}^\pm)$ solve the SDE \eqref{eq:slerhosde} with $(\kappa,0,\rho_+)$ in place of $(\kappa,\rho_-,\rho_+)$ and $(\tilde{\xi},\tilde{O}^\pm)$ in place of $(\xi,O^\pm)$. Let $\tilde{g}$ and $(\tilde{\zeta}_z)$ be the mapping out functions and swallowing times associated to $\tilde{\xi}$. Also let
	\begin{align*}
		T = \inf\{t\ge 0\colon \xi_t-O_t^-=C_\delta/2\}\quad\text{and}\quad \tilde{T} = \inf\{t\ge 0\colon \tilde{\xi}_t-\tilde{O}_t^-=C_\delta/2\}\;.
	\end{align*}
	Using Girsanov's theorem
	\begin{align*}
		\P(\zeta_1\ge t_\delta) &\le \P(T\le t_\delta) + \P(\zeta_1\ge t_\delta,T>t_\delta)\\
		&= \E(E_{t_\delta};\tilde{T}\le t_\delta) + \E(E_{t_\delta};\tilde{\zeta}_1\ge t_\delta, \tilde{T}>t_\delta) \\
		&\le \E(E_{t_\delta}^2)^{1/2}\left( \P(\tilde{T}\le t_\delta)^{1/2}+  \P(\tilde{\zeta}_1\ge t_\delta)^{1/2} \right)
	\end{align*}
	where
	\begin{align*}
		E_{t_\delta} &= \exp\left(\int_0^{t_\delta\wedge \tilde{T}} \frac{\rho_-/\sqrt{\kappa}}{\tilde{\xi}_s-\tilde{O}_s^-}\,dB_s - \frac{1}{2} \int_0^{t_\delta\wedge \tilde{T}} \left(\frac{\rho_-/\sqrt{\kappa}}{\tilde{\xi}_s-\tilde{O}^-_s}\right)^2\,ds \right) \quad\text{and hence}\\
		E_{t_\delta}^2 &= \exp\left(\int_0^{t_\delta\wedge \tilde{T}} \frac{2\rho_-/\sqrt{\kappa}}{\tilde{\xi}_s-\tilde{O}_s^-}\,dB_s - \frac{1}{2} \int_0^{t_\delta\wedge \tilde{T}} \left(\frac{2\rho_-/\sqrt{\kappa}}{\tilde{\xi}_s-\tilde{O}^-_s}\right)^2\,ds +  \int_0^{t_\delta\wedge \tilde{T}} \left(\frac{\rho_-/\sqrt{\kappa}}{\tilde{\xi}_s-\tilde{O}^-_s}\right)^2\,ds \right)\\
		&\le \exp\left(\frac{t_\delta \rho_-^2}{(C_\delta/2)^2\kappa}\right)\exp\left(\int_0^{t_\delta\wedge \tilde{T}} \frac{2\rho_-/\sqrt{\kappa}}{\tilde{\xi}_s-\tilde{O}_s^-}\,dB_s - \frac{1}{2} \int_0^{t_\delta\wedge \tilde{T}} \left(\frac{2\rho_-/\sqrt{\kappa}}{\tilde{\xi}_s-\tilde{O}^-_s}\right)^2\,ds \right)\;.
	\end{align*}
	From optional stopping we deduce that $\E(E_{t_\delta}^2)\le \exp(t_\delta(2\vee \epsilon^{-1})^2/((C_\delta/2)^2\epsilon))$. Then using that $\rho_+\le 0$ we get that a.s.
	\begin{align*}
		\tilde{g}_t(1)-\tilde{\xi}_t &= 1-\sqrt{\kappa}B_t -\int_0^t \frac{\rho_+\,ds}{\tilde{\xi}_s-\tilde{O}^+_s} + \int_0^t \frac{2\,ds}{\tilde{g}_s(1)-\tilde{\xi}_s} \\
		&\le 1-\sqrt{\kappa}B_t + \int_0^t \frac{(2+\rho_+)\,ds}{\tilde{g}_s(1)-\tilde{\xi}_s}\quad\text{for all $t<\tilde{\zeta}_1$}\;.
	\end{align*}
	By comparison theorems for SDE we can couple $\tilde{g}(1)-\tilde{\xi}$ with a Bessel process $X$ of dimension $\delta=1+2/\kappa\cdot (2+\rho_+)$ starting from $1/\sqrt{\kappa}$ such that $\tilde{g}(1)-\tilde{\xi}\le \sqrt{\kappa}X\le X/\sqrt{\epsilon}$ on the interval $[0,\tilde{\zeta}_1)$. Since $\delta\le 2(1-\epsilon)$ and $1/\sqrt{\kappa}\le 1/\sqrt{\epsilon}$, we can couple $X$ with a Bessel process $X'$ of dimension $2(1-\epsilon)$ started from $1/\sqrt{\epsilon}$ such that $X\le X'$ a.s. Hence
	\begin{align*}
		\P(\tilde{\zeta}_1\ge t_\delta) \le \P(X'/\sqrt{\epsilon}>0 \text{ on } [0,t_\delta))\;.
	\end{align*}
	Moreover, since $\tilde{O}^-$ is non-increasing, $\tilde{O}^-_0\le -C_\delta$ and $\tilde{\xi}\ge \sqrt{\kappa}B$ (the last statement is immediate from the defining SDE) we get
	\begin{align*}
		\P(\tilde{T}\le t_\delta) &=\P(\inf\nolimits_{[0,t_\delta]} (\tilde{\xi}-\tilde{O}^-)\le C_\delta/2) \le \P(\inf\nolimits_{[0,t_\delta]}\tilde{\xi}\le -C_\delta/2)\\
		&\le \P(\inf\nolimits_{[0,t_\delta]} B \le -C_\delta/(2\sqrt{\kappa})) \le  \P(\inf\nolimits_{[0,t_\delta]} B \le -C_\delta\sqrt{\epsilon}/2)\;.
	\end{align*}
	The result now follows by first choosing $t_\delta$ and then $C_\delta$ sufficiently large.
\end{proof}

\begin{proof}[Proof of Proposition \ref {prop:boosted-sle-conv}] 
	By Proposition \ref{prop:sle-rho-convergence} and Skorokhod's representation theorem, there are processes $(\tilde{\xi}^n,\tilde{O}^{n\pm})$ such that $(\xi^n,O^{n\pm})=_d (\tilde{\xi}^n,\tilde{O}^{n\pm})$ for all $n\le \infty$ and such that $(\tilde{\xi}^n,\tilde{O}^{n\pm})\to (\tilde{\xi}^\infty,\tilde{O}^{\infty\pm})$ as $n\to\infty$ uniformly on compacts a.s. Let $(\tilde{g}^n)$ and $(\tilde{\zeta}^n_z)$ be the mapping out functions and swallowing times associated to $(\tilde{\xi}^n)$. By Lemma \ref{lem:loewner-stuff} we have
	\begin{align}
		\label{eq:lower-loewner-swallowing}
		\tilde{\zeta}^\infty_x \le \liminf_{n\to \infty} \tilde{\zeta}^n_x\;.
	\end{align}
	Let $\tilde{T}^n_\epsilon = \inf\{t\ge 0\colon \tilde{g}^n(x)-\tilde{\xi}^n_t=\epsilon\}$ for $0<\epsilon<x/2$. Note first that $\tilde{T}^n_\epsilon < \tilde{\zeta}^n_x$. For fixed $n\le \infty$, $-\tilde{O}^{n-}$ and $\tilde{g}^n(x)$ are non-decreasing a.s.\ as can be seen from the defining SDE. Thus
	\begin{align*}
		\tilde{g}^n_{\tilde{T}^n_\epsilon}(x) - \tilde{\xi}^n_{\tilde{T}^n_\epsilon}&= \epsilon\;,\\
		\tilde{\xi}^n_{\tilde{T}^n_\epsilon} - \tilde{O}^{n-}_{\tilde{T}^n_\epsilon}&=  (\tilde{\xi}^n_{\tilde{T}^n_\epsilon} - \tilde{g}^n_{\tilde{T}^n_\epsilon}(x)) + \tilde{g}^n_{\tilde{T}^n_\epsilon}(x) - \tilde{O}^{n-}_{\tilde{T}^n_\epsilon} \ge -\epsilon+x\ge x/2
	\end{align*}
	on $\{\tilde{T}^n_\epsilon<\infty\}$ a.s. By the strong Markov property of $(\tilde{\xi}^n,\tilde{O}^{n\pm})$ at time $\tilde{T}^n_\epsilon$, its Brownian scaling property and Lemma \ref{lem:hit-boundary-fast}, we can therefore deduce that
	\begin{align*}
		\sup_{n\ge 1} \P(\tilde{\zeta}^n_x-\tilde{T}^n_\epsilon>\delta)\to 0\quad\text{as $\epsilon\to 0$ for all $\delta>0$}\;.
	\end{align*}
	Thus for $\delta>0$ and $\epsilon\in (0,x/2)$ we get
	\begin{align*}
		\P(\tilde{\zeta}^n_x- \tilde{\zeta}^\infty_x> 2\delta) &\le \P(\tilde{\zeta}^n_x- \tilde{T}^n_\epsilon> \delta) + \P(\tilde{T}^n_\epsilon-\tilde{\zeta}^\infty_x\ge \delta) \\
		&\le \sup_{m\ge 1} \P(\tilde{\zeta}^m_x-\tilde{T}^m_\epsilon>\delta) + \P(\tilde{T}^n_\epsilon>\tilde{T}^\infty_{\epsilon/2})\;.
	\end{align*}
	We have $\tilde{T}^n_\epsilon\ge \tilde{T}^\infty_{\epsilon/2}$ for only finitely many $n$ a.s.\ and by \eqref{eq:lower-loewner-swallowing} we thus get $\tilde{\zeta}^n_x\to \tilde{\zeta}^\infty_x$ in probability as $n\to \infty$. The claim follows.
\end{proof}

We will need one additional statement in the final stages of the proof of the main theorem, that we now state and prove:

\begin{cor}
	\label{cor:half-plane-dist-conv}
	Suppose that $\kappa^n \in (8/3,4]$ and $\rho^n\in (-2,\kappa^n/2-2)$ for $1\le n\le \infty$ be such that $(\kappa^n,\rho^n)\to (\kappa^\infty,\rho^\infty)$ as $n\to \infty$. For $n\le \infty$, let $\gamma^n \sim \SLE_{\kappa^n}(0,\rho^n)$ and $\tilde{K}^n$ be the chordal hull associated to $(\psi\circ\gamma^n)([0,\infty))$ where $\psi\colon \H\to \H$ is the Möbius transformation mapping $(0,-1,\infty)$ to $(0,\infty,1)$. Then $\hcap(\tilde{K}^n) \to \hcap(\tilde{K}^\infty)$ in distribution as $n\to \infty$.
\end{cor}

Before proving this, let us first state and prove the following lemma:

\begin{lemma}
	\label{lem:brownian-loopsoup-error}
	For all $\epsilon>0$ there exists $C_\epsilon>0$ such that the following is true: Let $\kappa\in (8/3,4]$, $\rho\in (-2,0]$ and $\eta\sim \SLE_\kappa(0,\rho)$. Also consider the Möbius transformation $\psi\colon \H\to \H$ which maps $(0,-1,\infty)$ to $(0,\infty,1)$ and we let $\tilde{K}$ be the chordal hull associated to $(\psi\circ\eta)([0,\infty))$. Then $\P(\hcap(\tilde{K})\ge C_\epsilon) \le \epsilon$.
\end{lemma}

\begin{proof}[Proof of Lemma \ref {lem:brownian-loopsoup-error}]
	As stated in the discussion below \cite[Theorem 2.1]{werner-wu-cle-sle}, whenever $\rho,\rho'>-2$ and $\kappa,\kappa'\in (8/3,4]$ are such that $\kappa\le \kappa'$ and so that they satisfy
	\begin{align}
		\label{eq:domination-eq}
		(\rho+2)\left(\frac{\rho+6}{\kappa}-1\right)\le (\rho'+2)\left(\frac{\rho'+6}{\kappa'}-1\right)\;,
	\end{align}
	one can couple $\eta_-\sim \SLE_{\kappa'}(0,\rho')$ with $\eta_+\sim \SLE_{\kappa}(0,\rho)$ such that $\eta_-$ lies left of $\eta_+$ (note that in \cite{werner-wu-cle-sle} the force point is located left of the starting point in contrast to the formulation here). We take $\kappa'=4$ and $\rho'$ sufficiently large so that the identity \eqref{eq:domination-eq} above holds for all $\kappa\in (8/3,4]$ and $\rho\le 0$. Let $L_\pm$ denote the chordal hull associated to $(\psi\circ \eta_\pm)([0,\infty))$. Then
	\begin{align*}
		\P(\hcap(\tilde{K})\ge C_\epsilon) = \P(\hcap(L_+)\ge C_\epsilon) \le \P(\hcap(L_-)\ge C_\epsilon)\;.
	\end{align*}
	The result follows since $\hcap(L_-)<\infty$ a.s.
\end{proof}

\begin{proof}[Proof of Corollary \ref {cor:half-plane-dist-conv}]
	By Proposition \ref{prop:boosted-sle-conv} and Skorokhod’s representation theorem, for each $x>0$ we can can consider a coupling of the curves $(\gamma^n)$ such that $\xi^n\to \xi^\infty$ uniformly on compacts a.s.\ and $\zeta^n_x\to \zeta^\infty_x$ a.s.\ as $n\to \infty$ where $\xi^n$ is the driving function of $\gamma^n$ and $\zeta^n$ its swallowing time of the point $x$ whenever $n\le \infty$. Therefore, writing $(g^n)$ for the mapping out functions associated to the driving functions $(\xi^n)$,
	\begin{align*}
		\int_0^{\zeta^n_x} \left(\frac{(g^n)_t'(-1)}{(g^n_t(-1)-\xi^n_t)^2}\right)^2\,dt \to \int_0^{\zeta^\infty_x} \left(\frac{(g^\infty)_t'(-1)}{(g^\infty_t(-1)-\xi^\infty_t)^2}\right)^2\,dt \quad\text{a.s.\ as $n\to \infty$}\;.
	\end{align*}
	Moreover, by Lemma \ref{lem:change-hcap} we get
	\begin{align*}
		E^n_x &:=\hcap(\tilde{K}^n) - \int_0^{\zeta^n_x} \left(\frac{(g^n)_t'(-1)}{(g^n_t(-1)-\xi^n_t)^2}\right)^2\,dt = \int_0^\infty \left(\frac{(g^n)_{t+\zeta^n_x}'(-1)}{(g^n_{t+\zeta^n_x}(-1)-\xi^n_{t+\zeta^n_x})^2}\right)^2\,dt \\
		&= (g^n_{\zeta^n_x})'(-1)^2\int_0^\infty \left(\frac{(g^n_{t+\zeta^n_x}\circ (g^n_{\zeta^n_x})^{-1})'(g^n_{\zeta^n_x}(-1))}{((g^n_{t+\zeta^n_x}\circ (g^n_{\zeta^n_x})^{-1})(g^n_{\zeta^n_x}(-1))-\xi^n_{t+\zeta^n_x})^2}\right)^2\,dt\;.
	\end{align*}
	Since $0\le (g^n_{\zeta^n_x})'(-1)\le 1$ and
	\begin{align*}
		\xi^n_{\zeta^n_x} - g^n_{\zeta^n_x}(-1) \ge \lim_{t\to \zeta^n_x} g_t(x) -g^n_{\zeta^n_x}(-1) \ge x+1
	\end{align*}
	we obtain that $(x+1)^{-2}\hcap(\tilde{K}^n)$ stochastically dominates the error term $E^n_x$; indeed, this follows from the strong Markov property of $\SLE_{\kappa^n}(0,\rho^n)$ and again Lemma \ref{lem:change-hcap}. This stochastic domination implies by Lemma \ref{lem:brownian-loopsoup-error} that for all $\delta>0$,
	\begin{align*}
		\sup_{n\ge 1}\P(E^n_x>\delta) \to 0\quad\text{as $x\to\infty$}\;.
	\end{align*}
	The result then follows directly.
\end{proof}

\section{The case of generalized SLEs with force points} 
\label{sec:gensle}

In some sense, the results in the previous section allow us to control how the law of the trunk of a $\CLE_\kappa$ exploration depends on $\kappa$ and an asymmetry parameter (in particular, when we consider the limit $\kappa \uparrow 4$).

We now turn our focus to the study of $\SLE_\kappa^\beta (\kappa-6)$ (for $\kappa < 4$) and $\SLE_4^{\langle \mu\rangle}(-2)$ processes which are the $\CLE_\kappa$ exploration processes themselves. The Loewner chain definition of these processes relies on local times and the excursion theory of Bessel processes. In order to control the dependence in $\kappa$, $\beta$ and $\mu$, we first have to review carefully these constructions.

Except for the very end of this section, we will derive facts of Bessel processes without making any reference to SLE curves.  

\subsection {Background on Bessel processes and their local times} 

The usual reference on the subject of Bessel processes is \cite[Chapter XI]{revuz-yor} but we will present a mostly self-contained treatment here; in particular, we will use the construction of the local times at $0$ given in \cite{donati-yor-bessel} (this forms the first part of the proof of Proposition \ref{prop:localdef}) to construct the whole local time field using stochastic calculus. We then have to make use of two constructions of compensated integrals which can be found in \cite{shef-cle, werner-wu-explorations}.

For the reader with detailed knowledge of Bessel processes, the results in this subsection (i.e.\ up to Proposition \ref{prop:excursion-int}) might appear standard but it seemed useful (or even necessary) to provide a complete exposition of the different constructions of compensated integrals and their equivalence (not least because the multiplicative convention of local times plays a key role in our setting).

Let $B$ be a standard Brownian motion and $\delta > 0$. For $x_0\ge 0$ and $y_0=x_0^2$ the following SDE (called \emph{squared Bessel SDE}) has a unique strong solution
\begin{align*}
	\begin{cases}
		dY(\delta)_t = 2\sqrt{|Y(\delta)_t|}\,dB_t + \delta\,dt\;,\\
		Y(\delta)_0 = y_0\;.
	\end{cases}
\end{align*}
In fact, $Y(\delta)\ge 0$ a.s.\ and we will work with a version for which this holds surely. The law of $Y(\delta)$ is called the \emph{squared Bessel process} of dimension $\delta$ started from $y_0$. We let $X(\delta)= \sqrt{Y(\delta)}$ and call the law of $X(\delta)$ a \emph{Bessel process} of dimension $\delta$ started from $x_0$ and we will always work with the above coupling to a Brownian motion $B$. If $\sigma(\delta)_r = \inf\{t\ge 0\colon Y(\delta)_t = r\}$ for $0<r<y_0$, then by Itô's formula, $(Y(\delta)^{\sigma(\delta)_r})^{1-\delta/2}$ is a local martingale and this easily implies that $Y(\delta)$ hits $0$ a.s.\ when $\delta<2$ and does not hit $0$ a.s.\ when $\delta\ge 2$ (and $y_0>0$). For later reference, we note that for $t>0$, the law of $X(\delta)_t$ (with $X(\delta)$ started from $x_0$) has a density $p^\delta_t(x_0,\cdot)$ on $(0,\infty)$ given by
\begin{align}
	\label{eq:besseldens}
	\begin{split}
		p^\delta_t(x_0,x)&= \frac{(x/x_0)^{\delta/2-1}x}{t}\,e^{-(x^2+x_0^2)/(2t)}\,I_{\delta/2-1}\left(\frac{x\cdot x_0}{t}\right)\quad\text{for $x_0>0$}\;,\\
		p^\delta_t(0,x) &= \frac{x^{\delta-1}}{2^{\delta/2-1}\, t^{\delta/2}\,\Gamma(\delta/2)}\,e^{-x^2/(2t)}\;.
	\end{split}
\end{align}
Here $I_{\delta/2-1}$ is the modified Bessel function of the first kind with index $\delta/2-1$. This result can for instance be found in the remarks below \cite[Chapter XI, Definition 1.9]{revuz-yor}. Let us now construct local times of Bessel processes. As mentioned above, the construction of $\ell(\delta)$ in the proof below is the one given in \cite{donati-yor-bessel}.

\begin{prop}
	\label{prop:localdef}
	Suppose that $\delta\in (0,2)$. Then $X(\delta)^{2-\delta}$ is a semimartingale with a local time process $l$ such that $(z,t)\mapsto l^z(\delta)_t$ is locally Hölder continuous of index $1/2-\epsilon$ for all $\epsilon>0$ and non-decreasing in $t$ for each fixed $x\ge 0$. We let
	\begin{align*}
		L^x(\delta) := \frac{l^{x^{2-\delta}}(\delta)}{2-\delta}=\frac{ (X(\delta)^{2-\delta} - x^{2-\delta})^+ - (x_0^{2-\delta}- x^{2-\delta})^+}{1-\delta/2} - 2\int_0^{\,\cdot} 1(X(\delta)_t>x)X(\delta)_t^{1-\delta}\,dB_t
	\end{align*}
	whenever $x\ge 0$. In fact, the stochastic integral defines an $L^2$ martingale. Moreover, the following occupation time formula holds: Almost surely
	\begin{align*}
		\int_0^T \psi(X(\delta)_t)\,dt = \int_0^\infty \psi(x)\,L^x(\delta)_T\,x^{\delta-1}\,dx\quad\text{for all measurable $\psi\colon [0,\infty)\to [0,\infty]\,,\ T\ge 0$}\;.
	\end{align*}
\end{prop}

\begin{proof}
	We will begin by showing that $X(\delta)^{2-\delta}$ is a semimartingale and establish its decomposition into a local martingale and finite variation part. By Itô's formula for $\epsilon>0$,
	\begin{align*}
		d(Y(\delta)_t+\epsilon)^{1-\delta/2} = (2-\delta)\sqrt{Y(\delta)_t}(Y(\delta)_t+\epsilon)^{-\delta/2}\,dB_t + \delta(1-\delta/2) \,\frac{\epsilon\,dt}{(Y(\delta)_t+\epsilon)^{1+\delta/2}}\;.
	\end{align*}
	From \eqref{eq:besseldens} we see that for any $T\ge 0$ we can bound
	\begin{align*}
		\E \left\langle \int_0^{\,\cdot} (2-\delta)Y(\delta)_t^{(1-\delta)/2}\,dB_t \right\rangle_{T} &= (2-\delta)^2\,\E\int_0^{T} Y(\delta)_t^{1-\delta}\,dt \\
		&=  (2-\delta)^2\, \int_0^T \E\left(X(\delta)_t^{2(1-\delta)}\right)\,dt < \infty\;.
	\end{align*}
	Therefore by stochastic dominated convergence, we deduce that
	\begin{align*}
		\int_0^{\,\cdot} (2-\delta) \sqrt{Y(\delta)_t}(Y(\delta)_t+\epsilon)^{-\delta/2}\,dB_t \to (2-\delta)\int_0^{\,\cdot} Y(\delta)_t^{(1-\delta)/2}\,dB_t =\vcentcolon M(\delta)^0
	\end{align*}
	u.c.p. as $\epsilon\to 0$ and the previous computation implies that the stochastic integral in the limit is defined. Let $\ell(\delta)=(X(\delta)^{2-\delta}-x_0^{2-\delta} - M(\delta)^0)/(1-\delta/2)$ so that
	\begin{align*}
		\int_0^{\,\cdot} \delta \,\frac{\epsilon\,dt}{(Y(\delta)_t+\epsilon)^{1+\delta/2}} \to \ell(\delta)\quad\text{u.c.p. as $\epsilon\to 0$}
	\end{align*}
	which entails that $\ell(\delta)$ is a non-decreasing continuous process; furthermore we can also observe that $1(Y(\delta)>\eta)\,d\ell(\delta)=0$ a.s.\ for all $\eta>0$ and hence $1(X(\delta)>0)\,d\ell(\delta)=0$ a.s. Thus $X(\delta)^{2-\delta}$ is a semimartingale and we can associate local times to it
	\begin{align*}
		l^z(\delta) &= 2\left( (X(\delta)^{2-\delta} - z)^+ - (x_0^{2-\delta}-z)^+ - \int_0^{\,\cdot} 1(X(\delta)^{2-\delta}_t>z)\,d\left(X(\delta)^{2-\delta}_t\right)\right)\\
		&= 2\left( (X(\delta)^{2-\delta} - z)^+ - (x_0^{2-\delta}-z)^+ - (2-\delta)\int_0^{\,\cdot} 1(X(\delta)^{2-\delta}_t>z)\,X(\delta)_t^{1-\delta}\,dB_t\right)\quad
	\end{align*}
	for $z\ge 0$. By classical local time theory, see for instance \cite[Chapter VI]{revuz-yor}, $(t,z)\mapsto l^z(\delta)_t$ has a version that is locally Hölder continuous of exponent $1/2-\epsilon$ for all $\epsilon>0$ on $[0,\infty)\times [0,\infty)$ and such that $l^z(\delta)$ is non-decreasing for all $z\ge 0$. Also, we have the occupation time formula i.e. a.s.
	\begin{gather*}
		(2-\delta)^2\int_0^T \phi(X(\delta)^{2-\delta}_t)Y(\delta)_t^{1-\delta}\,dt = \int_0^\infty \phi(z)l^z(\delta)_T\,dz
	\end{gather*}
	for all measurable $\phi\colon [0,\infty)\to [0,\infty)$. The result is now immediate.
\end{proof}

We point out that the processes $(x,T)\mapsto L^x(\delta)_T$ and $(x,T)\mapsto \lambda^{\delta/2-1}L^{\sqrt{\lambda}x}(\delta)_{\lambda T}$ have the same law when $X(\delta)_0=0$ and $\lambda>0$. This scaling property follows by using the Brownian scaling property of Bessel processes together with the definitions in the proposition above.

From now on, we will consider $x_0=0$. By standard excursion theory as exposed for instance in \cite[Chapter 22]{kallenberg} applied to the process $X(\delta)^{2-\delta}$ together with Proposition \ref{prop:localdef} there exists an infinite measure $\nu_\delta$ on
\begin{align*}
	E=\{e\in C([0,\infty),\R)\colon e_0=0\;,\ e|_{[\zeta_e,\infty)}=0\}\quad\text{where}\quad \zeta_e=\inf\{t>0\colon e_t=0\}
\end{align*}
called the \emph{excursion measure of the Bessel process of dimension $\delta$} such that the point process
\begin{align*}
	\sum_{\ell\ge 0\colon \tau(\delta)_\ell>\tau(\delta)_{\ell-}} \delta_{\left(\ell, \,X(\delta)_{(\tau(\delta)_{\ell-}+\,\cdot\,)\wedge \tau(\delta)_\ell}\right)}
\end{align*}
is a Poisson point process with intensity $\lambda_+\otimes \nu_\delta$ on the space $[0,\infty)\times E$ where here $\tau(\delta)_{\ell}=\inf\{t\ge 0\colon L^0(\delta)_t>\ell\}$ and $\tau(\delta)_{\ell-}= \inf\{t\ge 0\colon L^0(\delta)_t\ge \ell\}$ and where $\lambda_+$ is the Lebesgue measure on $[0,\infty)$. In the following, we will also write $e$ also for the coordinate process on $E$.

Note that by Proposition \ref{prop:localdef}, $X(\delta)^{2-\delta}-(1-\delta/2)L^0(\delta)$ is an $L^2$ martingale and it is easy to deduce by optional stopping that
\begin{align*}
	\E(R)= x^{2-\delta}/(1-\delta/2)\quad\text{where}\quad R=L^0(\delta)_{\,\inf\{t\ge 0\colon X_t \,=\, x\}}\;.
\end{align*}
By excursion theory, we know that $R$ is exponentially distributed with parameter $\nu_\delta(\sup e\ge x)$ and therefore
\begin{align}
	\label{eq:bessel-norm}
	\nu_\delta(\sup e \ge x) = (1-\delta/2)\,x^{\delta-2}\quad\text{for}\quad x>0\;.
\end{align}
Finally, we will need to introduce a very mild generalization of Bessel processes. Fix $\beta\in [-1,1]$. We start with the Bessel process $X(\delta)$ and, independently for all excursions, we replace each excursion by its negative with probability $(1-\beta)/2$ and leave it unchanged with probability $(1+\beta)/2$. We therefore obtain a process $X(\delta,\beta)$ which we call \emph{an asymmetric Bessel process}. Then
\begin{align*}
	\sum_{\ell\ge 0\colon \tau(\delta)_\ell > \tau(\delta)_{\ell-}} \delta_{\left(\ell, \,X(\delta,\beta)_{(\tau(\delta)_{\ell-}+\,\cdot\,)\wedge \tau(\delta)_\ell}\right)}
\end{align*}
is a Poisson point process of intensity $\lambda_+\otimes \nu_\delta^\beta$ where
\begin{align*}
	\nu_\delta^\beta = \frac{1+\beta}{2}\,\nu_\delta + \frac{1-\beta}{2}\,(e\mapsto -e)_*\nu_\delta \;.
\end{align*}
In the following, we will always consider the coupling of a Bessel process with an asymmetric Bessel process as described here.

\begin{prop}
	\label{prop:asymmetriclocal}
	For any $\beta\in [-1,1]$ there exists a local time process $(t,x)\mapsto L^x(\delta,\beta)_t$ defined on $[0,\infty)\times (\R\setminus\{0\})$ which is locally Hölder continuous with exponent $1/2-\epsilon$ for all $\epsilon>0$, non-decreasing in $t$ for each $x$ and a.s.\ satisfies the occupation time formula
	\begin{align*}
		\int_0^T \psi(X(\delta,\beta)_t)\,dt = \int_\R \psi(x)L^x(\delta,\beta)_T\,|x|^{\delta-1}\,dx\quad\text{for all measurable $\psi\colon \R\to [0,\infty]\,,\ T\ge 0$}\;.
	\end{align*}
	Moreover, if we define $L^{0\pm}(\delta,\beta) = (1\pm \beta)/2\cdot L^0(\delta)$, then by considering a suitable version of $L(\delta,\beta)$, we have
	\begin{align*}
		\sup_{[0,T]} \frac{|L^{\pm x}(\delta,\beta) - L^{ 0\pm}(\delta,\beta)|}{|x|^{1-\delta/2-\epsilon}} \to 0\quad\text{as $x\downarrow 0$ for all $T\ge 0$, $\epsilon>0$}\;.
	\end{align*}
\end{prop}

\begin{proof}
	The process $L(\delta,\beta)$ defined on $[0,\infty)\times (\R\setminus \{0\})$ with the stated properties and satisfying the given occupation time formula is easy to construct from $L(\delta)$. To analyze the limiting behavior, it suffices by symmetry to consider the `$+$' case. The case $\beta=-1$ is trivial, so we assume that $\beta>-1$. We note that
	\begin{align*}
		&\sum_{\ell\ge 0\colon \tau(\delta)_\ell > \tau(\delta)_{\ell-}} \delta_{\left(\ell, \,X(\delta,\beta)_{(\tau(\delta)_{\ell-}+\,\cdot\,)\wedge \tau(\delta)_\ell}\right)}\,1\left( X(\delta,\beta)_{(\tau(\delta)_{\ell-}+\tau(\delta)_\ell)/2}>0 \right) \\
		&\qquad \stackrel{d}{=} 
		\sum_{\ell\ge 0\colon \tau(\delta)_\ell>\tau(\delta)_{\ell-}} \delta_{\left(2/(1+\beta)\,\cdot\, \ell, \,X_{(\tau(\delta)_{\ell-}+\,\cdot\,)\wedge \tau(\delta)_\ell}\right)}\;.
	\end{align*}
	since both sides are Poisson point processes of intensity $(1+\beta)/2\cdot \lambda_+\otimes \nu_\delta$. Therefore
	\begin{align*}
		(L^x(\delta,\beta)_{\tau(\delta)_\ell}\colon (x,\ell)\in (0,\infty)\times [0,\infty)) \stackrel{d}{=} (L^x(\delta)_{\tau(\delta)_{(1+\beta)/2\,\cdot\, \ell}}\colon (x,\ell)\in (0,\infty)\times [0,\infty))
	\end{align*}
	and the result then follows from the Hölder continuity property in Proposition \ref{prop:localdef}.
\end{proof}

\begin{remark}
	\label{remark:bm-local-equivalence}
	Note that $X(1,0)$ is a standard Brownian motion and by the occupation time formula, $(X(1,0),L^{0+}(1,0))$ has the same law as the pair consisting of a Brownian motion and its local time at $0$.
\end{remark}

In the following section, it will be important to make sense of the process $\int_0^{\,\cdot} 1/X(\delta,\beta)_t\,dt$. It turns out (as we will see) that this integral does not converge absolutely whenever $\delta\le 1$ and our objective will be to use Lévy compensation to define it in this case as well.

\begin{prop}
	\label{prop:compensatedint}
	For $\beta\in [-1,1]$ and $\delta\in (0,1)$ the following expression defines a continuous process: For $t\ge 0$, let
	\begin{align*}
		I(\delta,\beta)_t &= \int_0^\infty x^{-1}|x|^{\delta-1}(L^x(\delta,\beta)_t - L^{0+}(\delta,\beta)_t)\,dx \\
		&\qquad + \int_{-\infty}^0 x^{-1}|x|^{\delta-1} (L^x(\delta,\beta)_t - L^{0-}(\delta,\beta)_t)\,dx\;.
	\end{align*}
	Moreover, when $\beta=1$ then $I(\delta,\beta)= 2/(\delta-1)\cdot (X(\delta)-B)$ almost surely. For $\delta=1$, one can also define a continuous process by
	\begin{align*}
		I(1,0)_t = \lim_{C\to \infty} \int_{-C}^{\,C} x^{-1}(L^x(1,0)_t - L^{0+}(1,0)_t)\,dx \quad\text{for $t\ge 0$}
	\end{align*}
	noting that the above is a limit of an eventually constant sequence. In both the $\delta<1$, and the $\delta=1$ and $\beta=0$ case, we have a.s.
	\begin{align}
		\label{eq:posstepintegral}
		I(\delta,\beta)_{t_+}- I(\delta,\beta)_{t_-} = \int_{t_-}^{t_+} 1/X(\delta,\beta)_s\,ds\quad\text{when $|X(\delta,\beta)|>0$ on $(t_-,t_+)$}\;.
	\end{align}
\end{prop}

We also define $I(1,\mu):=I(1,0) + \mu L^{0+}(1,0)$ when $\mu\in\R$ for later reference. Note that the definitions in the statement are very natural; indeed, if we apply the occupation time formula in Proposition \ref{prop:asymmetriclocal} we see for instance that
\begin{align*}
	\int_0^t 1(X(\delta,\beta)_t>0)/X(\delta,\beta)_t\,dt = \int_0^\infty x^{\delta-2}\,L^x(\delta,\beta)_t\,dx
\end{align*}
which diverges almost surely when $\delta\le 1$ (unless $\beta=-1$ of course) and is finite when $\delta>1$. The subtraction of the local time at $0$ in the definition creates the Lévy compensation needed to define the process in the $\delta\le 1$ case.

\begin{proof}[Proof of Proposition \ref{prop:compensatedint}]
	The fact that $I(\delta,\beta)$ is well-defined both when $\delta<1$ and $\beta\in [-1,1]$, and when $\delta=1$ and $\beta=0$ is immediate from the Hölder bound given in Proposition \ref{prop:asymmetriclocal}. Moreover, the occupation time formula given in Proposition \ref{prop:asymmetriclocal} also implies \eqref{eq:posstepintegral}. Suppose that $\delta<1$ and $\beta=1$, then by Itô's formula and \eqref{eq:posstepintegral}
	\begin{align*}
		F := I(\delta,\beta) - 2/(\delta-1)\cdot (X(\delta)-B)
	\end{align*}
	is constant on $(\tau(\delta)_{\ell-},\tau(\delta)_\ell)$ for all $\ell\ge 0$. Hence $F\circ \tau(\delta) = F\circ \tau(\delta)_-$ is continuous. Moreover, by the strong Markov property of $X(\delta)$, $F\circ \tau(\delta)$ has stationary and independent increments and so by Lévy's characterization, $F\circ \tau(\delta) = (\ell\mapsto \sigma W_\ell + c\ell)$ a.s.\ for some constants $\sigma\ge 0$, $c\in \R$ and a standard Brownian motion $W$. Thus $F= \sigma W\circ L^0(\delta) + cL^0(\delta)$ a.s. Since $F$ satisfies Brownian scaling, we necessarily have $\sigma=c=0$ as required (we use that $\delta<1$ here).
\end{proof}

\begin{prop}
	\label{prop:excursion-int}
	For $\delta\in (0,1]$ and $\epsilon>0$, we have
	\begin{align*}
		\nu_\delta\left(\int_0^{\zeta_e}\frac{dt}{e_t}=\infty\right)=0\quad\text{and}\quad \nu_\delta\left(\int_0^{\zeta_e}\frac{dt}{e_t}\ge \epsilon\right)=\epsilon^{\delta-2} \nu_\delta\left(\int_0^{\zeta_e}\frac{dt}{e_t}\ge 1\right)\;.
	\end{align*}
	If $\delta\in (0,1)$ and $\beta \in [-1,1]$ we can define a process $J(\delta,\beta, \epsilon)$ by
	\begin{gather*}
		J(\delta,\beta,\epsilon)_{l} = \sum_{\ell\le l\colon \tau(\delta)_\ell > \tau(\delta)_{\ell-}} \int_{\tau(\delta)_{\ell-}}^{\tau(\delta)_\ell} 1/X(\delta,\beta)_t\,dt \,1\left( \,\left|\int_{\tau(\delta)_{\ell-}}^{\tau(\delta)_\ell} 1/X(\delta,\beta)_t\,dt\right|\ge \epsilon\right) - C(\delta,\beta)_\epsilon l \\
		\text{where}\quad C(\delta,\beta)_\epsilon = \begin{cases}
			 \nu_\delta^\beta\left( \int_0^{\zeta_e}dt/e_t;\,\left|\int_0^{\zeta_e}dt/e_t\right| \ge \epsilon \right) &\colon \delta<1\;, \\
			 0 &\colon \delta=1\,,\;\beta=0\;.
		\end{cases}
	\end{gather*}
	Then $J(\delta,\beta,\epsilon)$ converges u.c.p. to the càdlàg process $I(\delta,\beta)\circ \tau(\delta)$ as $\epsilon\to 0$ in both the $\delta<1$ case and the case where $\delta=1$ and $\beta =0$.
\end{prop}

\begin{proof}
	The first two claims are immediate from Proposition \ref{prop:compensatedint} and the scaling property of $(X(\delta,\beta),L^0(\delta))$. For the convergence claim, note first that there exists a càdlàg process $J(\delta,\beta)$ such that $J(\delta,\beta,\epsilon)\to J(\delta,\beta)$ u.c.p. as $\epsilon\to 0$ (by excursion theory, this is just the usual construction of compensated stable Lévy processes). By \eqref{eq:posstepintegral} in Proposition \ref{prop:compensatedint}, $F:=J(\delta,\beta)-I(\delta,\beta)\circ\tau(\delta)$ is a.s.\ continuous. Moreover, by the strong Markov property of $X(\delta,\beta)$, $F$ has independent and stationary increments and hence by Lévy's characterization $F=(\ell\mapsto \sigma W_\ell+c\ell)$ a.s.\ for constants $\sigma\ge 0$, $c>0$ and a standard Brownian motion $W$. If $\delta<1$ then since $F$ satisfies stable scaling with exponent $2-\delta$ we get $\sigma=c=0$. If $\delta=1$ then since $F$ satisfies stable scaling with exponent $1$ we obtain $\sigma=0$ and since $\beta=0$ necessarily $c=0$ as required.
\end{proof}

\subsection{Continuity with respect to the dimension} 

After having reviewed these results on Bessel processes of dimension $\delta$, we now analyze the dependence of functionals of these processes with respect to $\delta$ (which is why we kept the dependence on $\delta$ in the notation above). More specifically, we will consider the $\delta\uparrow 1$ limit. The first main statement is the following (recall that we are only considering the case $x_0=0$): 

\begin{thm}
	\label{thm:bessel-int-conv}
	For any $\mu\in \R$ we have $(X(\delta,-\mu(1-\delta)/2),I(\delta,-\mu(1-\delta)/2)) \to (X(1,0),I(1,\mu))$ as $\delta\uparrow 1$ in distribution with respect to the topology of uniform convergence on compacts.
\end{thm}

The proof will be obtained by putting the following three lemmas together.

\begin{lemma}
	\label{lem:strong-bessel-conv}
	Let $\delta_n\uparrow 1$. Then $0\le X(1)-X(\delta_n)\downarrow 0$ and $L^0(\delta_n)\to L^0(1)$ u.c.p. as $n\to \infty$.
\end{lemma}

\begin{proof}
	By comparison theorems for SDEs (see \cite[Chapter IX, Theorem 3.7]{revuz-yor}) we have $Y(\delta_n)\le Y(1)$ for all $n\ge 1$ almost surely.
	Also a.s.\ for all $t\ge 0$
	\begin{align}
		\label{eq:bessel-tanaka-conv}
		\begin{split}
		Y(1)_t - Y(\delta_n)_t = 2\int_0^{t} (\sqrt{Y(1)_s}-\sqrt{Y(\delta_n)_s}\,)\,dB_s + (1-\delta_n)t\;.
		\end{split}
	\end{align}
	By optional stopping (the rigorous justification of this step is easy and left to the reader), we get that $\E(Y(1)_t-Y(\delta_n)_t) = (1-\delta_n)t$ for $t\ge 0$. Clearly the second term on the right-hand side of \eqref{eq:bessel-tanaka-conv} converges to $0$ uniformly on compacts a.s.\ as $n\to \infty$. For the first term on the right-hand side of \eqref{eq:bessel-tanaka-conv}, we use the BDG inequality to write
	\begin{align*}
		&\E\left( \sup_{[0,t]}\left( \int_0^{\,\cdot} (\sqrt{Y(1)_s}-\sqrt{Y(\delta_n)_s}\,) \,dB_s \right)^2 \right) \\
		&\quad \le 4\cdot \E\left( \int_0^t(\sqrt{Y(1)_s}-\sqrt{Y(\delta_n)_s}\,)^2\,ds\right) \\
		&\quad \le 4\int_0^t \E(Y(1)_s-Y(\delta_n)_s) \,ds \le 2(1-\delta_n)t^2\quad\text{for $t\ge 0$}\;.
	\end{align*}
	This yields the first claim. To get the convergence of the local times at $0$, note that by Proposition \ref{prop:localdef}, it suffices to show that
	\begin{align*}
		\int_0^{\,\cdot} 1(X(\delta_n)_s>0)X(\delta_n)_s^{1-\delta_n} \,dB_s \to \int_0^{\,\cdot} 1(X(1)_s>0)\,dB_s\quad\text{u.c.p. as $n\to\infty$}\;.
	\end{align*}
	Since $X(\delta_n)\le X(1)$, $X(\delta_n)^{1-\delta_n}\le 1+X(1)$ for all $n\ge 1$ a.s.\ and the claim follows from stochastic dominated convergence.
\end{proof}

The following lemma yields Theorem \ref{thm:bessel-int-conv} in the $\mu=0$ case (by considering the first and third marginals).

\begin{lemma}
	\label{lem:centred-bes-conv}
	We have $(X(\delta,0),L^0(\delta),I(\delta,0)) \to (X(1,0),L^0(1),I(1,0))$ as $\delta\uparrow 1$ in distribution with respect to the topology of uniform convergence on compacts.
\end{lemma}

\begin{proof}
	We have that by Lemma \ref{lem:strong-bessel-conv}, $X(\delta)\uparrow X(1)$ uniformly on compacts as $\delta\uparrow 1$ a.s.\ (the fact that the u.c.p.\ convergence implies uniform convergence a.s.\ follows from the monotonicity). Hence, for each excursion of $X(1)$,  there are excursions of $X(\delta)$ for each $\delta$, the endpoints of which approximate the endpoints of the excursion of $X(1)$ we started with. From Lemma \ref{lem:strong-bessel-conv} it is then easy to deduce that for all $\epsilon>0$,
	\begin{align*}
		\left(X(\delta,0), L^0(\delta),\int_0^{\,\cdot} \frac{1(|X(\delta,0)_t|>\epsilon)}{X(\delta,0)_t}\,dt\right) \stackrel{d}{\to} \left(X(1,0), L^0(1),\int_0^{\,\cdot} \frac{1(|X(1,0)_t|>\epsilon)}{X(1,0)_t}\,dt\right)
	\end{align*}
	as $\delta \uparrow 1$ with respect to the topology of uniform convergence on compacts. To deduce the result, it therefore suffices to prove that for all $T\ge 0$,
	\begin{align*}
		\E\left(\,\sup_{[0,T]} \left|  \int_0^{\,\cdot} \frac{1(|X(\delta,0)_t|>\epsilon)}{X(\delta,0)_t}\,dt - I(\delta,0) \right|\, \right) \to 0\quad\text{as $\epsilon\to 0$}
	\end{align*}
	uniformly in $\delta\in (0,1)$. By Proposition \ref{prop:compensatedint} for the first inequality and Proposition \ref{prop:asymmetriclocal} together with the definition of $I(\delta,1)$ for the second inequality below, we get
	\begin{align*}
		\E\left(\,\sup_{[0,T]} \left|  \int_0^{\,\cdot} \frac{1(|X(\delta,0)_t|>\epsilon)}{X(\delta,0)_t}\,dt - I(\delta,0) \right|\, \right) &\le 2\,\E\left(\,\sup_{[0,T]} \left|  \int_0^{\,\cdot} \frac{1(X(\delta)_t>\epsilon)}{X(\delta)_t}\,dt - I(\delta,1) \right|\, \right) \\
		&\le 2\int_{0}^\epsilon x^{\delta - 2}\,\E\left( \sup_{[0,T]} |L^x(\delta) - L^0(\delta)| \right)\,dx\;.
	\end{align*}
	By Proposition \ref{prop:localdef}, almost surely
	\begin{align*}
		L^x(\delta) - L^0(\delta) = \frac{-(x\wedge X(\delta))^{2-\delta}}{1-\delta/2} + 2 \int_0^{\,\cdot} 1(X(\delta)_t\in (0,x])X(\delta)_t^{1-\delta}\,dB_t\;.
	\end{align*}
	Now by the BDG inequality we get
	\begin{align*}
		\E\left( \sup_{[0,T]} |L^x(\delta) - L^0(\delta)| \right) &\lesssim \frac{x^{2-\delta}}{1-\delta/2} + \E\left(  \int_0^T 1(X(\delta)_t\in (0,x])X(\delta)_t^{2(1-\delta)}\,dt \right)^{1/2} \\
		&\le \frac{x^{2-\delta}}{1-\delta/2} + \left(  \int_0^T x^{2(1-\delta)}\P(X(\delta)_t\le x)\,dt \right)^{1/2}
	\end{align*}
	By the expression for the density of $X(\delta)_t$ given in \eqref{eq:besseldens}, $\P(X(\delta)_t\le x)\lesssim (x/\sqrt{t}\,)^\delta$ uniformly in $\delta\in (0,1)$. The claim is now immediate.
\end{proof}

\begin{lemma}
	\label{lem:constant-loc-bessel}
	Fix $\mu\ge 0$. Then $(-I(\delta,1)_{\tau(\delta)_{\mu(1-\delta)l}}\colon l\ge 0)\to (\mu l\colon l\ge 0)$  u.c.p. as $\delta\uparrow 1$. Moreover
	\begin{align*}
		\nu_\delta\left(\int_0^{\zeta_e} \frac{dt}{e_t} \ge 1\right)\to 1\quad\text{as $\delta\uparrow 1$}\;.
	\end{align*}
\end{lemma}

\begin{proof}
	Since the limit is deterministic, it suffices to show convergence in distribution. Let $\alpha=((1-\delta)\mu)^{1/(\delta-2)}$. From the definitions we have the scaling property
	\begin{gather*}
		(X(\delta),B,L^0(\delta)) \stackrel{d}{=} (X(\delta)_{\alpha^2(\cdot)}/\alpha,B_{\alpha^2(\cdot)}/\alpha,L^0(\delta)_{\alpha^2(\cdot)}/\alpha^{2-\delta})\;,\\
		\text{and hence}\quad (X(\delta),B,\tau(\delta)) \stackrel{d}{=} (X(\delta)_{\alpha^2(\cdot)}/\alpha,B_{\alpha^2(\cdot)}/\alpha,\tau(\delta)_{\alpha^{2-\delta}(\cdot)}/\alpha^2)
	\end{gather*}
	since $\tau(\delta)_l=\inf\{t\ge 0\colon L^0(\delta)_t>l\}$. Therefore by Proposition \ref{prop:compensatedint},
	\begin{align*}
		(-I(\delta,1)_{\tau(\delta)_{\mu(1-\delta)l}}\colon l\ge 0) &= \frac{2}{1-\delta}\,(X(\delta)_{\tau(\delta)_{\mu(1-\delta)l}} -B_{\tau(\delta)_{\mu(1-\delta)l}}\colon l\ge 0) \\
		&\stackrel{d}{=} \frac{2/\alpha}{1-\delta}(X(\delta)_{\tau(\delta)_l}-B_{\tau(\delta)_l}\colon l\ge 0) \\
		&= 2(1-\delta)^{(\delta-1)/(2-\delta)}\mu^{1/(2-\delta)}(X(\delta)_{\tau(\delta)_l}-B_{\tau(\delta)_l}\colon l\ge 0) \;.
	\end{align*}
	Since $(1-\delta)^{(\delta-1)/(2-\delta)}\mu^{1/(2-\delta)}\to \mu$ as $\delta\uparrow 1$ it suffices to show $(X(\delta_n)_{\tau(\delta_n)_l}-B_{\tau(\delta_n)_l}\colon l\ge 0)\to (l/2\colon l\ge 0)$ u.c.p. when $\delta_n\uparrow 1$. Fix $l'\ge 0$. Then
	\begin{align*}
		\sup_{l\in [0,l']} |X(\delta_n)_{\tau(\delta_n)_l} - B_{\tau(\delta_n)_l} - l/2| &= \sup_{[0,\tau(\delta_n)_{l'}]} |X(\delta_n)-B - L^0(\delta_n)/2| \\
		&\le \sup_{[0,\tau(\delta_n)_{l'}]} \left(|X(\delta_n)-X(1)| + |L^0(\delta_n)-L^0(1)|/2\right)\quad\text{a.s.}
	\end{align*}
	where the inequality follows from Proposition \ref{prop:localdef} which entails that $X(1)-B = L^0(1)/2$ a.s. It is not difficult to see that $L^0(1)_t\to \infty$ as $t\to\infty$ a.s.; this together with the u.c.p. convergence $L^0(\delta_n)\to L^0(1)$ as $n\to \infty$ from Lemma \ref{lem:strong-bessel-conv} then readily implies that $(\tau(\delta_n)_{l'})$ is tight and the result then follows from Lemma \ref{lem:strong-bessel-conv}.
	
	To deduce the additional result, by excursion theory and the first part of the lemma with $\mu=l=1$, we obtain
	\begin{align*}
		\exp\left((1-\delta)\nu_\delta\left( e^{-i\int_0^{\zeta_e}dt/e_t}-1+i\int_0^{\zeta_e}dt/e_t \right)\right)=\E\left(e^{-iI(\delta,1)_{\tau(\delta)_{1-\delta}}}\right)\to e^i\quad\text{as $\delta\uparrow 1$ }\;.
	\end{align*}
	By Proposition \ref{prop:excursion-int} hence
	\begin{align*}
		\nu_\delta\left( \int_0^{\zeta_e} \frac{dt}{e_t} \ge 1 \right) (1-\delta)\int_0^\infty x^{\delta-2}(1-e^{-ix})\,dx \to 1\quad\text{as $\delta\uparrow 1$}\;.
	\end{align*}
	Note that
	\begin{align*}
		(1-\delta)\int_0^\infty x^{\delta-2}(1-e^{-ix})\,dx =  (1-\delta)\int_0^1 x^{\delta-2}(1-e^{-ix})\,dx +  1 - (1-\delta)\int_1^\infty x^{\delta-2}e^{-ix}\,dx\;.
	\end{align*}
	The first term on the right-hand side goes to $0$ as $\delta\uparrow 1$ since $|1-e^{-ix}|\lesssim x$ for $x>0$ and the second integral on the right-hand side also goes to $0$ as $\delta\uparrow 1$ as can be seen by looking at the cancellations between the contributions at $x$ and $x+\pi$ to the integral. The result follows.
\end{proof}

\begin{proof}[Proof of Theorem \ref{thm:bessel-int-conv}]
	By symmetry, it suffices to consider the case $\mu\le 0$. Let $\beta = -\mu(1-\delta)/2$ for $\delta$ sufficiently close to $1$. We couple $X(\delta,\beta)$ and $X(\delta,0)$ as follows: For each negative excursion of $X(\delta,0)$ we multiply the excursion by $-1$ with probability $\beta$ (and leave it unchanged otherwise) independently for each such negative excursion. By this excursion theory description
	\begin{align*}
		\P(\sup\nolimits_{\,[0,T]} |X(\delta,\beta) - X(\delta,0)|&\ge \epsilon ) - \P(T>\tau(\delta)_l) \le \P(\sup\nolimits_{\,[0,\tau(\delta)_l]} |X(\delta,\beta) - X(\delta,0)|\ge \epsilon ) \\
		&= 1-e^{-\beta/2\cdot \nu_\delta(\sup e \ge \epsilon)l} = 1-e^{-\beta/2\cdot (1-\delta/2)\epsilon^{\delta-2}l}
	\end{align*}
	using \eqref{eq:bessel-norm} to obtain the last equality. Since $L^0(\delta)\to L^0(1)$ u.c.p. as $\delta \uparrow 1$ we obtain that $X(\delta,-\mu(1-\delta)/2) - X(\delta,0)\to 0$ u.c.p. as $\delta \uparrow 1$. 
	By the construction of this coupling we see that
	\begin{align*}
		\sum_{\ell\ge 0\colon \tau(\delta)_\ell > \tau(\delta)_{\ell-}} \delta_{\left(\ell, \,X(\delta,\beta)_{(\tau(\delta)_{\ell-}+\,\cdot\,)\wedge \tau(\delta)_\ell}\right)} - \sum_{\ell\ge 0\colon \tau(\delta)_\ell > \tau(\delta)_{\ell-}} \delta_{\left(\ell, \,X(\delta,0)_{(\tau(\delta)_{\ell-}+\,\cdot\,)\wedge \tau(\delta)_\ell}\right)} \stackrel{d}{=} P_\delta^+ - P_\delta^-
	\end{align*}
	where $P_\delta^+$ is a Poisson point processes of intensity $\beta/2\cdot \lambda_+\otimes \nu_\delta$ and $P_\delta^-=(e\mapsto -e)_*P_\delta^+$. Therefore by the excursion description of the compensated integrals given in Proposition \ref{prop:excursion-int},
	\begin{gather}
		\label{eq:i-delta-conv}
		\begin{split}
		&\quad (I(\delta,\beta)_{\tau(\delta)_l} - I(\delta,0)_{\tau(\delta)_l}\colon l\ge 0) \stackrel{d}{=} (I(\delta,1)_{\tau(\delta)_{\beta l}}\colon l\ge 0)\quad\text{for $\delta<1$}\\
		&\text{and hence}\quad (I(\delta,-\mu(1-\delta)/2)_{\tau(\delta)_l} - I(\delta,0)_{\tau(\delta)_l}\colon l\ge 0) \to (\mu l/2\colon l\ge 0)\quad\text{u.c.p. as $\delta\uparrow 1$}
		\end{split}
	\end{gather}
	by the first part of Lemma \ref{lem:constant-loc-bessel}. Let $T(\delta)_t = \inf\{s\ge t\colon X(\delta)_s=0\}$ for $t\ge 0$. We have $\tau(\delta)\circ L^0(\delta) = T(\delta)$ a.s.\ for $\delta<1$ and hence by \eqref{eq:i-delta-conv} and $L^0(\delta)\to L^0(1)$ u.c.p. as $\delta\uparrow 1$,
	\begin{align*}
		(I(\delta,-\mu(1-\delta)/2)_{T(\delta)_t} - I(\delta,0)_{T(\delta)_t} - \mu L^0(\delta)_t/2\colon t\ge 0) \to 0\quad\text{u.c.p. as $\delta\uparrow 1$}\;.
	\end{align*}
	By Lemma \ref{lem:centred-bes-conv} we therefore only need to establish that
	\begin{align}
		\label{eq:bessel-final}
		\begin{split}
		A(\delta)&:= (I(\delta,-\mu(1-\delta)/2)_{T(\delta)_t} - I(\delta,0)_{T(\delta)_t}\colon t\ge 0) \\
		&\qquad - (I(\delta,-\mu(1-\delta)/2)_{t} - I(\delta,0)_{t}\colon t\ge 0) \to 0
		\end{split}
	\end{align}
	u.c.p. as $\delta\uparrow 1$. Note that (recalling the coupling of $X(\delta,0)$ and $X(\delta,-\mu(1-\delta)/2)$ as defined at the beginning of the proof),
	\begin{align*}
		A(\delta)_t = 2\cdot 1(X(\delta,-\mu(1-\delta)/2)_t>0>X(\delta,0)_t)\,\int_t^{T(\delta)_t} \frac{ds}{X(\delta)_s}\quad\text{for all $t\ge 0$ a.s.}
	\end{align*}
	Fix $\epsilon>0$ and $T\ge 0$. Then for $l\ge 0$ we get by excursion theory
	\begin{align*}
		\P(\sup\nolimits_{\,[0,T]} &A(\delta)\ge \epsilon) - \P(T>\tau(\delta)_l ) \le \P(\sup\nolimits_{\,[0,\tau(\delta)_l]} A(\delta)\ge \epsilon) \\
		&= 1- \exp\left(-\frac{-\mu(1-\delta)}{2}\, l\cdot \nu_\delta\left(\int_0^{\zeta_e}ds/e_s\ge \epsilon/2\right)\right) \to 0\quad\text{as $\delta\uparrow 1$}
	\end{align*}
	by Proposition \ref{prop:excursion-int} and the second part of Lemma \ref{lem:constant-loc-bessel}. The claim \eqref{eq:bessel-final} now follows since $L^0(\delta)\to L^0(1)$ u.c.p. as $\delta\uparrow 1$.
\end{proof}

\subsection {Reformulation in terms of generalized SLE processes}

So far in this section, we have not discussed the SLE aspects of all these results on Bessel processes. In the remaining few paragraphs, we will recall the definition of generalized SLE with force points as first presented in \cite{shef-cle}; we also refer to this paper and \cite{werner-wu-explorations} for the motivation behind the definition. In the following, recall the definition of the processes $X(\delta,\beta)$, $I(\delta,\beta)$ for $\delta<1$ and $I(1,\mu)$ as introduced above.

Suppose that $\kappa\in (8/3,4)$ and $\beta\in [-1,1]$. Let
\begin{align*}
	O(\kappa,\beta) = \frac{-2}{\sqrt{\kappa}}\,I(3-8/\kappa,\beta)\quad\text{and}\quad \xi(\kappa,\beta) = O(\kappa,\beta) + \sqrt{\kappa}\,X(3-8/\kappa,\beta)\;.
\end{align*}
$\SLE_\kappa^\beta(\kappa-6)$ is the Loewner chain associated to the driving function $\xi(\kappa,\beta)$. It is a very non-trivial result established in \cite[Theorem 7.4]{cle-percolations} that this Loewner chain is generated by a continuous curve. Analogously, if $\kappa=4$ and $\mu\in \R$ we let
\begin{align*}
	O(4,\mu) = - I(1,-\mu)\quad\text{and}\quad \xi(4,\mu) = O(4,\mu) + 2 X(1,0)\;.
\end{align*}
$\SLE_4^{\langle \mu\rangle}(-2)$ is the Loewner chain associated to $\xi(4,\mu)$ and this Loewner chain is generated by a continuous curve as stated in \cite[Proposition 5.3]{cle-percolations}. Note that the sign convention for $\mu$ is chosen so that $\xi(4,\mu)$ is non-decreasing in $\mu$ and that the definition given here matches the one from the introduction by Remark \ref{remark:bm-local-equivalence}. Crucially, Theorem \ref{thm:bessel-int-conv} now immediately implies the following result.

\begin{prop}
	\label{prop:general-sle-conv}
	Consider $\mu\in \R$. Then in the notation above,
	\begin{align*}
		(\xi(\kappa,\mu(4/\kappa-1)),O(\kappa,\mu(4/\kappa-1)))\to (\xi(4,\mu),O(4,\mu))
	\end{align*}
	in distribution with respect to uniform convergence on compacts as $\kappa\uparrow 4$.
\end{prop}

\section{Wrapping up the proof of the main result}
\label{sec:mainresult}

This final section is devoted to the proof of the main result Theorem \ref{thm:main-result}. We begin by mentioning a consequence of \cite{cle-percolations} and \cite{msw-simple}. This proposition is illustrated in Figure \ref{fig:bcle-duality}. Note that this is the step where the very non-trivial relation between the asymmetry parameter in the driving function of the exploration process of a $\CLE_\kappa$ and the driving function of its trunk in the case $\kappa\in (8/3,4)$ enters the picture.

\begin{prop}
	\label{prop:bcle-duality-import}
	Consider $\kappa\in (8/3,4)$, $\beta\in [-1,1]$ and let $\gamma\sim \SLE_\kappa^\beta(\kappa-6)$; moreover, let $\gamma_L$ and $\gamma_R$ be the left and right boundary of $\gamma$ respectively. Let $\kappa'=16/\kappa$ and define $\rho'=\rho'(\kappa,\beta)\in [\kappa'-6,0]$ to be the unique value such that
	\begin{align*}
		\tan(\pi\rho'/2)= \frac{\sin(\pi\kappa'/2)}{1+\cos(\pi\kappa'/2)-2/(1-\beta)}
	\end{align*}
	and let $ \tilde{\rho}_L=-2-\kappa\rho'/4$ and $\tilde{\rho}_R=3\kappa/2-6+\kappa\rho'/4$. 
	Also consider $\eta\sim \SLE_{\kappa'}(\rho'(\kappa,\beta),\kappa'-6-\rho'(\kappa,\beta))$ and conditionally on $\eta$, sample two conditionally independent curves $(\eta_L,\eta_R)$ from $0$ to $\infty$ as follows:
	\begin{itemize}
		\item Let $\eta_R$ be the concatenation of independent $\SLE_\kappa(\tilde{\rho}_R,0)$ curves in each of the right complementary components of $\eta$ that intersect the real line (always from the leftmost real point to the rightmost real point of the complementary component).
		\item Similarly let $\eta_L$ be the concatenation of independent $\SLE_\kappa(0,\tilde{\rho}_L)$ curves in the left complementary components of $\eta$ that intersect the real line (in this case always from the rightmost to the leftmost real point in each component).
	\end{itemize}
	Then $(\gamma_L,\gamma_R)$ and $(\eta_L,\eta_R)$ have the same law.
\end{prop}

\begin{figure}
	\centering
	\def\svgwidth{0.8\columnwidth}
	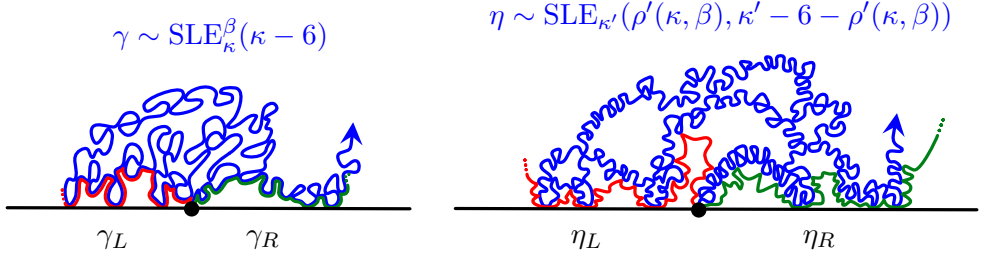
	\caption{Illustration of Proposition \ref{prop:bcle-duality-import}. \emph{Left.} $\gamma_L$ and $\gamma_R$ are the left and right boundaries of $\gamma$ respectively. \emph{Right.} We first sample $\eta$ and then $\eta_L$ (resp. $\eta_R$) as a concatenation of $\SLE_\kappa(0,\tilde{\rho}_L)$ curves (resp. $\SLE_\kappa(\tilde{\rho}_R,0)$ curves) in the boundary touching complementary components of $\eta$ left of (resp. right of) $\eta$.}
	\label{fig:bcle-duality}
\end{figure}

\begin{proof}
	Let
	\begin{align}
		\label{eq:relations}
		\rho_R := -\frac{\kappa}{4}\,(\rho'+2)\quad\text{and}\quad \rho_L:=\frac{\kappa}{2}-4-\rho_R = -\frac{\kappa}{4}\left((\kappa'-6-\rho')+2\right)
	\end{align}
	and note that $\rho_L=\kappa-6-\tilde{\rho}_L$ and $\rho_R=\kappa-6-\tilde{\rho}_R$. Conditionally on $\eta$, sample two conditionally independent curves $(\nu_L,\nu_R)$ from $0$ to $\infty$ as follows:
	\begin{itemize}
		\item In each of the right complementary components of $\eta$ intersecting the real line, draw an independent $\SLE_{\kappa'}(\rho',0)$ from the rightmost to the leftmost real point of the complementary component. In each of the complementary components between this curve and $\eta$, we draw an independent $\SLE_\kappa(\kappa-6-\rho_R,\rho_R)$ from the first to the last point on $\eta$ that is in the complementary component. We call the concatenation of all these simple curves $\nu_R$.
		\item Similarly, in each of the left complementary components of $\eta$ intersecting the real line, we first draw an $\SLE_{\kappa'}(0,\kappa'-6-\rho')$ from the leftmost to the rightmost boundary point; in each of the complementary components between this curve and $\eta$, we sample an independent $\SLE_\kappa(\rho_L,\kappa-6-\rho_L)$ curve from the first to the last point on $\eta$ that is in the particular complementary component. The concatenation of all the simple curves so constructed is called $\nu_L$.
	\end{itemize}
	It is a direct consequence of \cite[Theorem 7.4]{cle-percolations} (the inexplicit parameter there is given in \cite[Theorem 1.6]{msw-simple}) that $(\gamma_L,\gamma_R)$ has the same law as $(\nu_L,\nu_R)$. In the terminology of \cite{cle-percolations}, we are describing the construction of the first layer of loops of a $\BCLE_{\kappa'}(\rho')$. The fact that $(\nu_L,\nu_R)$ and $(\eta_L,\eta_R)$ have the same law, now follows from Lemma \ref{lem:ig-main-iteration} and \eqref{eq:relations}. 
\end{proof}

\begin{cor}
	\label{cor:bcle-input-stopped}
	Consider $\gamma\sim \SLE_\kappa^\beta(\kappa-6)$ from $0$ to $\infty$ in $\H$ and $\gamma_*\sim \SLE_\kappa^\beta(\kappa-6)$ from $0$ to $1$ in $\H$ where $\kappa\in (8/3,4)$ and $\beta\in [-1,1]$. Moreover set $\kappa'=16/\kappa$ and recall the definition of $\rho'(\kappa,\beta)$ from Proposition \ref{prop:bcle-duality-import}. Let $\tau$ be the swallowing time of $1$ for the curve $\gamma$ and $K_*$ be the chordal hull associated to $\gamma_*([0,\infty))$. Let $(\xi,O^\pm)$ be the processes defining $\eta\sim \SLE_{\kappa'}(\rho',\kappa'-6-\rho')$ as in Section \ref{subsec:sle-force-points} and let $\zeta_1$ be the swallowing time for the point $1$ by $\eta$. Also let $\nu\sim \SLE_\kappa(0,-2-\kappa\rho'/4)$ from $0$ to $1$ in $\H$ be independent of $\eta$. Then
	\begin{align*}
		2\tau \stackrel{d}{=} \hcap(K_*) \stackrel{d}{=} 2\zeta_1 + (O^+_{\zeta_1}-O^-_{\zeta_1})^2 \hcap(L)
	\end{align*}
	where $L$ is the chordal hull associated to $\nu([0,\infty))$.
\end{cor}

\begin{proof}
	The first equality in distribution is clear from the target invariance of $\SLE_\kappa^\beta(\kappa-6)$ which says that $\gamma$ and $\gamma_*$ can be coupled so that they agree on $[0,\tau]$, see \cite[Proposition 3.14]{shef-cle}.
	
	Also let $\eta_*\sim \SLE_{\kappa'}(\rho',\kappa'-6-\rho')$ in $\H$ start at $0$ and end at $1$ and write $K_0$ for the chordal hull associated to $\eta_*([0,\infty))$ and let $a_\pm(K_0)$ be as in Section \ref{subsec:loewner-chains}. We write $K$ for the chordal hull associated to $g_{K_0}^{-1}\circ (a_-(K_0)+(a_+(K_0)-a_-(K_0))\nu)([0,\infty))$. Then by applying the Möbius transformation mapping $(-1,0,\infty)$ to $(\infty,0,1)$ to the result in Proposition \ref{prop:bcle-duality-import} we obtain that $K$ and $K_*$ have the same law and in particular $\hcap(K)$ and $\hcap(K_*)$ have the same law. The target invariance property of $\SLE_{\kappa'}(\rho',\kappa'-6-\rho')$ (see \cite{sw-coord}) curves implies that $\eta$ and $\eta_*$ can be coupled so they agree on $[0,\zeta_1]$ and in particular
	\begin{align*}
		(\hcap(K_0),a_-(K_0),a_+(K_0)) \stackrel{d}{=} (2\zeta_1,O^-_{\zeta_1}, O^+_{\zeta_1})\;.
	\end{align*}
	The claim now follows from standard additivity and scaling properties of half-plane capacities.
\end{proof}

Another input to the proof of the main theorem is the analogous proposition for $\kappa=4$ which involves the non-explicit parameter appearing in Theorem \ref{thm:main-result-msw} that we wish to determine. For the statement, recall the definition of $M$ in Theorem \ref{thm:main-result-msw}.

\begin{prop}
	\label{prop:critical-bcle-key}
	Consider $\rho\in (-2,0)$, let $\mu=M(\rho)$ and sample $\gamma\sim \SLE_4^{\langle\mu\rangle}(-2)$. Let $\gamma_L$ and $\gamma_R$ be the left and right boundaries of $\gamma$ respectively. Also let $\tilde{\rho}_L=-2-\rho$, $\tilde{\rho}_R=\rho$ and consider $\eta\sim \SLE_4(\rho,-2-\rho)$. Conditionally on $\eta$, we sample two conditionally independent curves $(\eta_L,\eta_R)$ as follows:
	\begin{itemize}
		\item Let $\eta_R$ be the concatenation of $\SLE_4(\tilde{\rho}_R,0)$ curves in each of the right complementary components of $\eta$ (from the leftmost to the rightmost real point of the component).
		\item  Let $\eta_L$ be the concatenation of $\SLE_4(0,\tilde{\rho}_L)$ curves in each of the left complementary components of $\eta$ (from the rightmost to the leftmost real point of the component).
	\end{itemize}
	Then $(\gamma_L,\gamma_R)$ and $(\eta_L,\eta_R)$ have the same law.
\end{prop}

\begin{proof}
	This is immediate from Theorem \ref{thm:main-result-msw} and \cite[Section 5]{cle-percolations}.
\end{proof}

\begin{cor}
	\label{cor:bcle-critical-input-stopped}
	Consider $\gamma\sim \SLE_4^{\langle\mu\rangle}(-2)$ from $0$ to $\infty$ in $\H$ and $\gamma_*\sim \SLE_4^{\langle\mu\rangle}(-2)$ from $0$ to $1$ in $\H$ where $\rho\in (-2,0)$ and $\mu=M(\rho)$. We let $\tau$ be the swallowing time of $1$ for the curve $\gamma$ and $K_*$ be the chordal hull associated to $\gamma_*([0,\infty))$. Moreover, let $(\xi,O^\pm)$ be the processes defining $\eta\sim \SLE_4(\rho,-2-\rho)$ as in Section \ref{subsec:sle-force-points} and let $\zeta_1$ be the swallowing time for the point $1$ by $\eta$. Also let $\nu\sim \SLE_4(0,-2-\rho)$ from $0$ to $1$ in $\H$ be independent of $\eta$. Then
	\begin{align*}
		2\tau \stackrel{d}{=} \hcap(K_*) \stackrel{d}{=} 2\zeta_1 + (O^+_{\zeta_1}-O^-_{\zeta_1})^2 \hcap(L)
	\end{align*}
	where $L$ is the chordal hull associated to $\nu([0,\infty))$.
\end{cor}

\begin{proof}
	The proof is an immediate adaption of the proof of Corollary \ref{cor:bcle-input-stopped} and is thus omitted. The relevant target invariance properties also appear in \cite[Proposition 3.14]{shef-cle} and \cite{sw-coord}.
\end{proof}

A final item that we need in order to be ready for the proof of  Theorem \ref{thm:main-result} is the following monotonicity result for the laws of the $\CLE_4$ exploration paths, which  is based on the level line coupling of these exploration paths with a GFF.

\begin{lemma}
	\label{lem:monotone-msw}
	Suppose that $\mu_-<\mu_+$. Then there is a coupling of
	\begin{align*}
		\gamma_-\sim \SLE^{\langle\mu_-\rangle}_4(-2)\quad\text{and}\quad \gamma_+\sim \SLE^{\langle\mu_+\rangle}_4(-2)
	\end{align*}
	from $-i$ to $i$ in $\D$ such that the left boundary of $\gamma_-$ lies left of the left boundary of $\gamma_+$ and that these boundaries are different almost surely.
\end{lemma}

\begin{proof}
	Recall the coupling $(\eta_c,\gamma_c)$ for $c\in (-\lambda,\lambda)$ and $\Gamma\sim \CLE_4^0$ with a Dirichlet GFF $h$ in $\D$ from Section \ref{sec:introduction}. Let $\gamma_\pm := \gamma_{c_\pm}$ where $\rho_\pm = R(\mu_\pm)$ and $c_\pm = \lambda(1+\rho_\pm)$. By Theorem \ref{thm:main-result-msw} we have $\rho_-<\rho_+$ and hence $c_-<c_+$. Standard level line results (that can be viewed as the special $\kappa =4$ case of the more general results in \cite{ig1}) imply that $\eta_{c_-}$ lies left of $\eta_{c_+}$, and $\eta_{c_-}$ and $\eta_{c_+}$ are different a.s. The result now follows by the construction of $\gamma_{c_\pm}$ in terms of $(\eta_{c_\pm},\Gamma)$.
\end{proof}

The proof will use the results from Section \ref{subsec:sle-force-points} and Section \ref{sec:gensle} to obtain the limit of the distributional equalities appearing in Corollary \ref{cor:bcle-input-stopped} and concludes by comparing the resulting statement with Corollary \ref{cor:bcle-critical-input-stopped}. In particular, the function $M$ we are identifying in this work enters the proof only via Corollary \ref{cor:bcle-critical-input-stopped}.

\begin{proof}[Proof of Theorem \ref{thm:main-result}]
	Fix $\rho\in (-2,0)$ and let $\mu = M(\rho)$. Also define $\tilde{\mu}=-\pi\cot(\pi\rho/2)$. We will first show that $\mu\le \tilde{\mu}$. Fix some sequence $\kappa_n\uparrow 4$ and let
	\begin{align*}
		\beta_n =\tilde{\mu}(4/\kappa_n-1)\;,\quad \rho'_n = \rho'(\kappa_n,\beta_n)\quad\text{and}\quad \kappa_n'=16/\kappa_n
	\end{align*}
	recalling the definition of $\rho'(\cdot,\cdot)$ from Proposition \ref{prop:bcle-duality-import}. By Proposition \ref{prop:general-sle-conv} and Skorokhod's representation theorem we can couple curves with the following marginals
	\begin{align*}
		\gamma^n\sim \SLE_{\kappa_n}^{\beta_n}(\kappa_n-6)\quad\text{for $n<\infty$ and}\quad \tilde{\gamma}^\infty \sim \SLE_4^{\langle \tilde{\mu}\rangle}(-2)
	\end{align*}
	such that the driving function of $\gamma^n$ converges to the driving function of $\tilde{\gamma}^\infty$ uniformly on compacts almost surely as $n\to \infty$. Let $\tau_n$ (resp. $\tilde{\tau}^\infty$) denote the swallowing time of $1$ by $\gamma^n$ (resp. $\tilde{\gamma}^\infty$). Then by Lemma \ref{lem:loewner-stuff} we get
	\begin{align}
		\label{eq:liminfhitting}
		\tilde{\tau}^\infty \le \liminf_{n\to \infty} \tau^n\quad\text{a.s.}
	\end{align}
	We also let $\gamma^\infty\sim \SLE_4^{\langle \mu\rangle}(-2)$ and write $\tau^\infty$ for its swallowing time of $1$.
	
	Next, let $(\xi^n,O^{n\pm})$ be the process defining $\eta^n\sim \SLE_{\kappa_n'}(\rho_n',\kappa_n'-6-\rho_n')$ for $n<\infty$ and $(\xi^\infty,O^{\infty\pm})$ be the process defining $\eta^\infty \sim \SLE_4(\rho,-2-\rho)$ as in Section \ref{subsec:sle-force-points}. For $n\le \infty$, let $\zeta^n_1$ be the swallowing time of $1$ by $\eta^n$. Finally, let $\nu^n\sim \SLE_{\kappa_n}(0,-2-\kappa_n\rho_n'/4)$ from $0$ to $1$ in $\H$ be independent of $\eta^n$ for $n<\infty$ and $\nu^\infty\sim \SLE_4(0,-2-\rho)$ from $0$ to $1$ in $\H$ be independent of $\eta^\infty$. We write $K^n$ for the chordal hull associated to $\nu^n([0,\infty))$ for each $n\le \infty$. Then Corollary \ref{cor:bcle-input-stopped} and  Corollary \ref{cor:bcle-critical-input-stopped} state that
	\begin{align*}
		\tau^n \stackrel{d}{=} \zeta^n_1 + (O^{n+}_{\zeta^n_1}- O^{n-}_{\zeta^n_1})^2\hcap(K^n)/2\quad\text{for all $n\le \infty$}\;.
	\end{align*}
	It is easy to see that $\rho_n'\to \rho$ as $n\to \infty$ and hence by Proposition \ref{prop:boosted-sle-conv} and Corollary \ref{cor:half-plane-dist-conv} we obtain that
	\begin{align*}
		\zeta^n_1 + (O^{n+}_{\zeta^n_1}- O^{n-}_{\zeta^n_-})^2\hcap(K^n)/2 \stackrel{d}{\to} \zeta^\infty_1 + (O^{\infty+}_{\zeta^\infty_1}- O^{\infty-}_{\zeta^\infty_1})^2\hcap(K^\infty)/2 \quad\text{as $n\to \infty$}
	\end{align*}
	and thus $\tau^n$ converges to $\tau^\infty$ in distribution as $n\to \infty$. By \eqref{eq:liminfhitting} we may therefore deduce that $\tau^\infty$ stochastically dominates $\tilde{\tau}^\infty$.
	
	It follows from Lemma \ref{lem:monotone-msw} and the first distributional equality in Corollary \ref{cor:bcle-critical-input-stopped} that $\mu\le \tilde{\mu}$. We have established that $M(\rho)\le -\pi\cot(\pi\rho/2)$ for all $\rho\in (-2,0)$. Thus also by the symmetry statement in Theorem \ref{thm:main-result-msw} and the bound we have just established we get
	\begin{align*}
		M(\rho)=-M(-2-\rho)\ge \pi\cot(\pi(-2-\rho)/2)=-\pi\cot(\pi\rho/2)
	\end{align*}
	for all $\rho\in (-2,0)$. This completes the proof.
\end{proof}

\bibliographystyle{hmralphaabbrv}
\bibliography{cle-sle-trunk}

\end{document}